\documentclass{article}
\usepackage{graphicx}%
\usepackage{multirow}%
\usepackage{mathrsfs}%
\usepackage[title]{appendix}%
\usepackage{xcolor}%
\usepackage{textcomp}%
\usepackage{manyfoot}%
\usepackage{booktabs}%
\usepackage{algorithm}%
\usepackage{algorithmicx}%
\usepackage{algpseudocode}%
\usepackage{listings}%
\usepackage[utf8]{inputenc}
\usepackage{amssymb,amsmath}
\usepackage{amsmath}
\usepackage{amssymb}
\usepackage{amsthm}
\usepackage{tikz-cd}
\usepackage{float}
\usepackage{todonotes}
\usepackage{xcolor}
\usepackage{comment}
\usepackage{enumitem}
\usepackage{physics}
\usepackage{url}
\usepackage{hyperref}
\usepackage{rotating,lscape}

\usepackage{tabularx}
\usepackage{booktabs}
\newcolumntype{Y}{>{\raggedleft\arraybackslash}X}%
\usepackage{pgf,pgfarrows,pgfnodes,tikz,color,tkz-berge,tkz-graph}
\usetikzlibrary{arrows,shapes,snakes,automata,backgrounds,petri,topaths,calc}

\allowdisplaybreaks

\newcommand{\proba}[1]{\ensuremath{{\Bbb P}\!\left[#1 \right]} }
\newtheorem{theorem}{Theorem}
\newtheorem{proposition}[theorem]{Proposition}

\newtheorem{lemma}[theorem]{Lemma}
\theoremstyle{definition}
\newtheorem{definition}[theorem]{Definition}

\newtheorem{remark}[theorem]{\bf Remark}

\newcommand{\RR}{\mathbb{R}}
\newcommand{\CC}{\mathbb{C}}
\newcommand{\NN}{\mathbb{N}}
\newcommand{\ZZ}{\mathbb{Z}}

\newcommand{\bbx}{\boldsymbol{x}}

\newcommand{\bbk}{\boldsymbol{k}}
\newcommand{\bbv}{\boldsymbol{v}}
\newcommand{\algsystem}{\mathcal{S}}
\newcommand{\solset}{\mathrm{Sol}}
\newcommand{\leader}{\mathrm{ld}}
\newcommand{\init}{\mathrm{init}}

\newcommand{\vect}[1]{\boldsymbol{#1}}

\newcommand{\D}[2]{ \ensuremath{ \frac{d #1 }{d #2 } } }

\usepackage{hyperref}

\title{Identifying Markov chain models from time-to-event data: an algebraic approach}

\author{
  Ovidiu Radulescu$^{1,*}$, Dima Grigoriev$^2$, Matthias Seiss$^3$, \\
  Maria Douaihy$^{1,4}$,
  Mounia Lagha$^4$,
  Edouard Bertrand$^5$,
  }

 \date{
    $^1$ LPHI, University of Montpellier and CNRS, Place Eugène Bataillon, Montpellier 34095, France
    \\
    $^2$ CNRS, Mathématiques, Université de Lille, 59655, Villeneuve d’Ascq, France \\
        $^3$ University of Kassel, Kassel, Germany \\
         $^4$ IGMM, University of Montpellier and CNRS, 1919 Rte de Mende, Montpellier 34090, France  
         \\
                $^5$ 
                IGH, University of Montpellier and CNRS, 141 Rue de la Cardonille, Montpellier 34094, France
                \\
 \bf \today
 }

\begin{document}

\maketitle

\abstract{  
 {
Many biological and medical questions can be modeled using time-to-event data in finite-state Markov chains, with the phase-type distribution describing intervals between events. We solve the inverse problem: given a phase-type distribution, can we identify the transition rate parameters of the underlying Markov chain? For a specific class of {\em solvable} Markov models, we show this problem has a unique solution up to finite symmetry transformations, and we outline a recursive method for computing symbolic solutions for these models across any number of states. Using the Thomas decomposition technique from computer algebra, we further provide symbolic solutions for any model.
Interestingly, different models with the same state count but distinct transition graphs can yield identical phase-type distributions. To distinguish among these, we propose additional properties beyond just the time to the next event. We demonstrate the method's applicability by inferring transcriptional regulation models from single-cell transcription imaging data.
}

{\bf Keywords}:Markov chains, phase-type distribution, inverse problem, symmetric polynomials, Thomas decomposition, 
transcriptional bursting.
}

\maketitle

\section{Introduction}
Continuous-time finite Markov chains have been largely used to model biological and medical phenomena
 {involving multiple discrete states
}\cite{bharucha1997elements}. For instance, the progression of a disease 
 {may have} multiple stages, including varying severity grades, recovery after treatment, and potential relapse. 
Similarly, gene expression functions intermittently, and often requires several prior non-productive and productive stages to prepare and initiate transcription or translation, or to pause these processes. In both examples, the number of  {states} and the timescales of transitions between them are usually unknown. 

 

 { A common dataset for such models consists of a sequence of specific events, such as the occurrence of an observed state. In renewal processes, relevant to many applications, the times separating successive events are independent, identically distributed variables, and the sequence is characterized by the distribution of these intervals \cite{feller1991introduction}. In simple cases, such as Poisson processes, this distribution is exponential. 
However, when multiple states are involved, this distribution deviates from the exponential. The renewal process scenario applies only when there is one observable state and the process resets to the same unobserved state after each observation.}
These scenarios are modeled by phase-type distributions, introduced by \cite{neuts1975probability}, 
(see Figure~\ref{figure1}).

Phase-type distributions were used 
 {in biology and medicine}
for modeling
transcription and translation bursting \cite{kumar_transcriptional_2015,soltani_intercellular_2016,kumar_constraining_2019,bokes_mixture_2020,Tantale2021,Pimmett2021,douaihy2023burstdeconv,zhang_exact_2024}, 
stochastic enzymatic reactions \cite{moffitt_extracting_2014},
ion channel dynamics \cite{lamar_reduction_2011},
drug kinetics \cite{faddy1993structured}, 
population genetics \cite{hobolth_phase-type_2019,hossjer_phase-type_2018},
and public health issues \cite{fackrell2009modelling,liquet2012investigating,asanjarani2021estimation,stone2022systematic}.


In this paper, our aim is to determine the underlying Markov chain, the number of states, and the transition rates between states from the phase-type distribution. This inverse problem has previously been tackled using maximum likelihood approaches, where the data likelihood was computed for an arbitrarily  chosen model and its maximization yielded optimal model parameters \cite{asanjarani2021estimation,stone2022systematic}.

In contrast to this direct inference approach, our method decomposes the inference problem into two parts. The first part involves regressing a parametric multi-exponential representation of the phase-type distribution \cite{dufresne2007fitting,douaihy2023burstdeconv}. 
The second part of our approach
is entirely algebraic. We formulate 
the inverse problem
as a system of polynomial equations
in the transition rate parameters
and solve this system for 
different numbers of states and
transition topologies. In doing so, 
we obtain explicit formulas relating
transition rate parameters to phase-type distribution parameters
for  all Markov chains that generate the same multi-exponential distribution.

Other approaches to finding  
Markov chain representations of a phase-type
distribution, using generating functions or the Laplace transform \cite{maier_algebraic_1991, commault_invariant_1996}, propose some, but not all, Markov chain representations and do not always offer explicit formulas for the transition rate parameters.



We provide a code implementation of our algorithms that formulate the inverse problem of any phase-type distribution as a system of polynomial equations and solve it using Thomas decomposition, a computer algebra method. Our approach also determines whether the model is solvable, meaning whether the solution to the inverse problem is unique or unique up to transformations by finite symmetries.

To demonstrate the relevance of our approach, we apply it to the context of transcriptional bursting, a well-documented phenomenon in gene expression research \cite{Dufourt2018,Tantale2021,Pimmett2021,Bellec2022,douaihy2023burstdeconv,damour2023transcriptional}. 

The structure of the paper is as follows: In Section 2, we discuss the phase-type distribution problem and find its solution in the non-degenerate case, specifically when the generator of the Markov chain has distinct eigenvalues. Section 3 introduces the symbolic inverse problem, which entails finding the transition rate parameters from the phase-type distribution parameters. We demonstrate that solving this inverse problem involves a system of polynomial equations symmetric in the phase-type distribution parameters. Additionally, we discuss solvable models in Section 3, where the inverse problem has a unique solution.
In Section 4, we present an example of a solvable model with an arbitrary number of states. Section 5 utilizes Thomas decomposition to solve the inverse problem. Section 6 addresses model degeneracy issues. Finally, in Section 7, we apply our findings to analyze transcriptional bursting data.

\begin{figure}[h!]
\begin{center}
\scalebox{0.5}{
\begin{tikzpicture}
 \SetUpEdge[lw         = 0.5pt,
            color      = black,
            labelstyle = {sloped,scale=2}]
  \SetVertexMath
       \tikzset{VertexStyle/.style={scale=1,
       draw,
            shape = circle,
            line width = 1pt,
            color = black,
            outer sep=1pt}}
  \Vertex[x=0,y=4]{1}
  \Vertex[x=3,y=10]{4}
  \Vertex[x=3,y=6]{3}
  \Vertex[x=6,y=4]{2}
\tikzset{EdgeStyle/.style={post, bend right = 40,line width = 2.5}}
\Edge[label=$k_1$](1)(3)
\tikzset{EdgeStyle/.style={post, bend right = 40,line width = 2.5}}
\Edge[label=$k_3$](3)(1)
\tikzset{EdgeStyle/.style={post, bend right = 40,line width = 2.5}}
\Edge[label=$k_2$](2)(3)
\tikzset{EdgeStyle/.style={post, bend right = 40,line width = 2.5}}
\Edge[label=$k_4$](3)(2)
\tikzset{EdgeStyle/.style={post,line width = 2.5}}
\Edge[label=$k_5$](3)(4)
\tikzset{EdgeStyle/.style={dashed,post,bend right = 40,line width = 2.5}}
\Edge[label={\tiny return}](4)(3)
\end{tikzpicture}
}
\hfill
\includegraphics{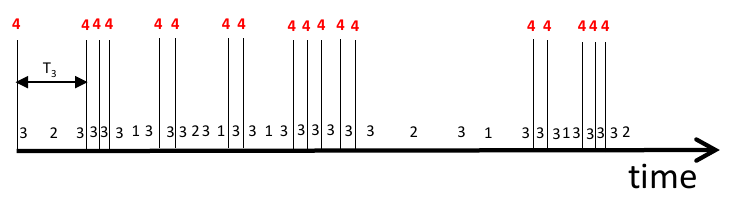}


\end{center}
\caption{\label{figure1} 
 {
Phase-type distribution for a three-state Markov chain model with state $4$ as the observable state. A typical dataset records the moments when $4$
is reached. The chain returns instantly from  $4$
to $3$ and restarts. The interval between successive observations, also defined as the time $T_3$ to reach $4$ from $3$, follows a phase-type distribution.
} 
 }
\end{figure}

\section{Phase-type distribution, direct problem}\label{direct_problem}

\subsection{Statistical definition}
Consider the successive observations of  {a state},  {as illustrated in}  Figure~\ref{figure1}. 
 {Following each observation, the system returns to 
a state $s$, so the interval between successive 
observations corresponds to the time required to reach the observed
state from $s$, denoted
by $T_s$}.
The distribution of $T_s$ is of the {\em phase-type}  {if it is derived from}
convolutions and mixtures (convex combinations) of
exponential distributions \cite{commault_phase-type_2003}.
 {Alternatively,
it can be defined as having a Laplace transform that is a rational function (see} Theorem 1 in  
\cite{commault_phase-type_2003}).
Phase-type distributions generalize exponential 
distributions,  {including mixtures of exponential or Erlang 
distributions. }


 {This paper considers multi-exponential distributions, characterized by a complementary cumulative distribution function (survival function) of the following type:}
\begin{equation}\label{multiexp}
S(t) = \proba{T_s > t}  = \sum_{i=1}^{N} A_i \exp (\lambda_i t),
\end{equation}
where $\lambda_i < 0$ are all negative, and  
$\sum_{i=1}^{N} A_i = 1$.
$A_i$ are not necessarily all positive, but are 
constrained by the conditions $S(t)  > 0, S'(t) < 0$
for all $t \geq 0$.

This family of distributions includes mixtures of
exponentials when all $A_i$ are positive but is not restricted to this case (some $A_i$ can be negative). 
The memoryless exponential distribution
corresponds to the case $N=1$,  {while cases with  $N>1$} correspond
to phase-type distributions with memory.
 {Focusing on} multi-exponential distributions 
is  not restrictive,
  {as} any phase-type distribution can be approximated with arbitrary precision by distributions of  {this} type. 
Furthermore, as will be
 {shown} in the next section, finite Markov chains generically
lead to multi-exponential distributions.

\subsection{Finite state continuous time Markov chains}
In this subsection, we introduce phase-type distributions using Markov chain representations and demonstrate the conditions under which we obtain the multi-exponential distributions described by \eqref{multiexp}. 
 {We assume that the reader is familiar with the basic theory of finite-state continuous-time Markov chains and refer to \cite{bharucha1997elements,meyn2012markov} for excellent presentations of this subject. 
We also use the well known result from operational
research that any phase-type distribution can be defined as the distribution of the absorption time of a finite-state, continuous-time Markov chain \cite{commault_phase-type_2003, bladt_review_2005}.}

 {To formalize the problem,
we consider a continuous time Markov chain $M(t)$, where 
the first $N$ states are unobservable and the state $N+1$
is the observed state.}



The chain is defined by its generator $\vect{Q}$, an $(N+1) \times (N+1)$ matrix with entries  $Q_{ij}>0$ representing 
the transition rate from the state $i$ to the state $j$ for $j\neq i$, and $Q_{ii} =  -\sum_{j\neq i} Q_{ij}$ (zero row sum).


We are interested in the distribution of $T_s$, the first passage time to state $N+1$ starting from the state $s$ with $1\leq s \leq N$:
$$T_s = \inf \{ t \, | \,  t> 0, \, M(t) = N+1, \, M(0) = s \},$$
where $M(t)$ is the state of the chain at time  $t$.


 {Let us consider two similar models.
In the {\em model with return}, 
the chain returns instantly to state $s$ once reaching $N+1$
($Q_{N+1,s}= \infty$) and the process restarts.  
Since $s$ is the unique return state, the observation timings 
form a renewal process with inter-event times following the distribution of $T_s$ (see \cite{feller1991introduction} for an introduction to renewal processes). }
In the {\em model with absorption}, 
$N+1$ is absorbing and $Q_{N+1,i}=0$ for all $i \neq N+1$.  
In this case, $T_s$ is the time to absorption starting from $s$.

 {The distribution of $T_s$
  is the same for both models but suited to different applications: the absorption model applies when observations stop after the event, like in survival analysis or gene fixation, while the return model fits recurring events, such as transcription bursting or epidemiology. Even when the return model is contextually appropriate, we use the equivalent absorption model for distribution calculations due to technical considerations.

If several return states are 
possible, then the model with return is no longer a renewal process. The distribution of the inter-event times depends on the last return state and is no longer a phase-type distribution. }

 {As examples of Markov models we have 
generated all the models having $N$ unobservable states and satisfying the two conditions:
\begin{itemize}
\item 
The matrix $\vect{Q}$ has $2N-1$ independent, non-diagonal elements (kinetic parameters). This assumption is necessary to establish a one-to-one relation between the parameters of the $T_s$
distribution and the model's kinetic parameters, ensuring a unique solution to the inverse problem.
\item
Only the state $N$ leads to $N+1$,
 a simplifying assumption that fits the application discussed and aids in presentation. General results without this assumption are presented in the 
 Subsection~\ref{sec:noC1}.
\end{itemize}
For $N=3$, up to permutations of the states, there are nine models that satisfy these assumptions. We label them as $M1$
to $M9$. Among these models, only five are ergodic, i.e. 
any state can be reached from any other state through transitions of the Markov chain. We will see in Section~\ref{sec:solv} that ergodic models are interesting because they can yield a unique solution to the inverse problem. }

With the absorption assumption, the generator of the Markov chain model $M9$ that has $N=3$, represented in the Figure~\ref{figure2} is
\begin{equation}
\vect{Q} = 
\begin{pmatrix}
-k_1 &  0 & k_1 & 0 \\
0 & -k_2 & k_2 & 0 \\
k_3 & k_4 & -(k_3+k_4+k_5) & k_5 \\
0 & 0 & 0 & 0   
\end{pmatrix}.
\end{equation}



For the calculation of the distribution of $T_s$ it is convenient to introduce the state probabilities 
$X_i(t)  = \proba{M(t) = i \, | \, M(0) = s}$ for $1\leq i \leq N+1$ . 
The variables $X_i (t)$ satisfy the following system of linear differential
equations (the master equation):
\begin{equation}
\D{\vect{X}}{t} = \vect{Q}^T \vect{X} {(t)},
\label{diffeq}
\end{equation}
with the initial conditions $X_n(0) = \delta_{n,s}$, where $\delta_{n,s}$ is the Kronecker symbol and $\vect{Q}^T$ 
(called dual generator)
is the transpose of the generator matrix $\vect{Q}$, and 
$\vect{X}=(X_1,\ldots,X_N,X_{N+1})^T$.
Because the last column of $\vect{Q}^T$ is zero ($N+1$ is absorbing), the variables $X_1,\ldots,X_N$ satisfy an autonomous ODE system.
Indeed, let $\vect{\tilde Q}$ be the $N \times N$ matrix obtained by eliminating the last line and the last
column of $\vect{Q}^T$. For the example considered, we have
\begin{equation}
\vect{\tilde Q} = 
\begin{pmatrix}
-k_1 &  0 & k_3  \\
0 & -k_2 & k_4  \\
k_1 & k_2 & -(k_3+k_4+k_5)
\end{pmatrix}.
\end{equation}
Then $\vect{\tilde{X}}(t) = (X_1(t),\ldots,X_N(t))^T$ satisfies
\begin{equation}\label{eq:probaevolution}
\D{\vect{\tilde X}}{t} =  \vect{\tilde Q} \vect{\tilde X} {(t)},
\end{equation}
with initial conditions $X_i(0) = \delta_{i,s}$, where $\delta_{i,s}$
is the Kronecker delta.

 {Unless stated otherwise, we assume that the state 
$N+1$ can be reached only from state $N$.
Thus, the} remaining variable $X_{N+1}$ satisfies 
\begin{equation}\label{eq:probaevolutionN}
\D{X_{N+1}}{t} = Q_{N,N+1} X_N {(t)},
\end{equation}
with the initial condition $X_{N+1}(0) = 0$.

Let us consider that the eigenvalues $\lambda_i, 1\leq i \leq N$ of the matrix $\vect{\tilde Q}$ satisfy the following {\em non-degeneracy condition}
\begin{equation}\label{nondegeneracy}
\lambda_i \ne \lambda_j, \text{ for all } 1\leq i \neq j \leq N,  \text{ and } \max_{1\leq i \leq N} (\lambda_i ) < 0.
\end{equation}
\begin{remark}
It can be shown that 
$\max_{1\leq i \leq N} (\lambda_i ) \leq 0$  is always satisfied. Indeed, $X_i(t), 1\leq i \leq N$ are probabilities, therefore remain bounded for all $t$, which means that one cannot have $\lambda_i>0$ for some $1\leq i \leq N$. If some $\lambda_i>0$, then there are solutions of \eqref{eq:probaevolution} that increase without bound when $t$ increases. 
We will show later that strong connectedness  of the transition graph 
of the Markov chain
reduced to the vertices $1,\ldots,N$ implies that none of the eigenvalues 
$\lambda_i$ can be zero.
One should note that, unlike $\vect{Q}$,  $\vect{\tilde{Q}}^T$ is not a continuous time Markov chain generator and does not satisfy the zero row sum rule.
\end{remark}
Considering that the condition \eqref{nondegeneracy} is satisfied, the
solutions of \eqref{eq:probaevolution},\eqref{eq:probaevolutionN} read
\begin{equation}
\vect{\tilde{X}}(t) = \sum_{i=1}^N C_i \vect{u}_i e^{\lambda_i t } \quad \mathrm{and} \quad X_{N+1}(t) = Q_{N,N+1}  \sum_{i=1}^N \frac{C_i  u_{i,N}}{\lambda_i} (e^{\lambda_i t } -1),
\label{solution}
\end{equation}
where $\lambda_i$ and $\vect{u}_i=(u_{i1},\ldots,u_{i,N})^T$ for $1 \leq i \leq N$ are eigenvalues and eigenvectors of $\vect{\tilde Q}$, respectively.

Because $N+1$ is absorbing, 
$$X_{N+1}(t)  = \proba{M(t) = {N+1} \, | \, M(0) = s} = \proba{T_s \leq t}$$ 
is the
cumulative distribution function of the time $T_s$.
We also define the survival function of $T_s$
as follows
\begin{equation}
S(t) = \proba{T_s > t} = 1 - X_{N+1}(t) = \sum_{i=1}^N A_i e^{\lambda_i t},
\label{survival}
\end{equation}
where
\begin{equation} \label{Ai}
 A_i = - \frac{Q_{N,N+1} C_i u_{i,N}}{\lambda_i}.
\end{equation}
The constants $A_i$ satisfy 
\begin{equation}\label{sumA}
\sum_{i=1}^N A_i = 1,
\end{equation}
which follows from $S(0)=1$.

This shows that in the non-degenerate case the distribution of $T_s$ is multi-exponential
with $2N-1$ independent parameters $\lambda_1,\ldots,\lambda_N$ and $A_1,\ldots,A_{N-1}$.
\begin{figure}[h!]
\begin{center}
\scalebox{0.3}{
\begin{tikzpicture}
 \SetUpEdge[lw         = 0.5pt,
            color      = black,
            labelstyle = {sloped,scale=2}]
  \SetVertexMath
       \tikzset{VertexStyle/.style={scale=1,
       draw,
            shape = circle,
            line width = 1pt,
            color = black,
            outer sep=1pt}}
  \Vertex[x=0,y=4]{1}
  \Vertex[x=3,y=10]{4}
  \Vertex[x=3,y=6]{3}
  \Vertex[x=6,y=4]{2}
\tikzset{EdgeStyle/.style={post, bend right = 20,line width = 2.5}}
\Edge[label=$k_1$](1)(2)
\tikzset{EdgeStyle/.style={post, bend right = 10,line width = 2.5}}
\Edge[label=$k_3$](2)(1)
\tikzset{EdgeStyle/.style={post, bend right = 20,line width = 2.5}}
\Edge[label=$k_2$](1)(3)
\tikzset{EdgeStyle/.style={post, bend right = 20,line width = 2.5}}
\Edge[label=$k_4$](3)(1)
\tikzset{EdgeStyle/.style={post,line width = 2.5}}
\Edge[label=$k_5$](3)(4)
\tikzset{EdgeStyle/.style={dashed,post,bend right = 40,line width = 2.5}}
\Edge[label={\tiny return}](4)(3)
\end{tikzpicture}
}
\hfill
\scalebox{0.3}{
\begin{tikzpicture}
 \SetUpEdge[lw         = 0.5pt,
            color      = black,
            labelstyle = {sloped,scale=2}]
  \SetVertexMath
       \tikzset{VertexStyle/.style={scale=1,
       draw,
            shape = circle,
            line width = 1pt,
            color = black,
            outer sep=1pt}}
  \Vertex[x=0,y=4]{1}
  \Vertex[x=3,y=10]{4}
  \Vertex[x=3,y=6]{3}
  \Vertex[x=6,y=4]{2}
\tikzset{EdgeStyle/.style={post, bend right = 20,line width = 2.5}}
\Edge[label=$k_1$](1)(2)
\tikzset{EdgeStyle/.style={post, bend right = 20,line width = 2.5}}
\Edge[label=$k_3$](2)(1)
\tikzset{EdgeStyle/.style={post, bend left = 40,line width = 2.5}}
\Edge[label=$k_2$](1)(3)
\tikzset{EdgeStyle/.style={post, bend left = 40,line width = 2.5}}
\Edge[label=$k_4$](3)(2)
\tikzset{EdgeStyle/.style={post,line width = 2.5}}
\Edge[label=$k_5$](3)(4)
\tikzset{EdgeStyle/.style={dashed,post,bend right = 40,line width = 2.5}}
\Edge[label={\tiny return}](4)(3)
\end{tikzpicture}
}
\hfill
\scalebox{0.3}{
\begin{tikzpicture}
 \SetUpEdge[lw         = 0.5pt,
            color      = black,
            labelstyle = {sloped,scale=2}]
  \SetVertexMath
       \tikzset{VertexStyle/.style={scale=1,
       draw,
            shape = circle,
            line width = 1pt,
            color = black,
            outer sep=1pt}}
  \Vertex[x=0,y=4]{1}
  \Vertex[x=3,y=10]{4}
  \Vertex[x=3,y=6]{3}
  \Vertex[x=6,y=4]{2}
\tikzset{EdgeStyle/.style={post, bend right = 40,line width = 2.5}}
\Edge[label=$k_1$](1)(2)
\tikzset{EdgeStyle/.style={post, bend right = 40,line width = 2.5}}
\Edge[label=$k_3$](2)(3)
\tikzset{EdgeStyle/.style={post, bend right = 40,line width = 2.5}}
\Edge[label=$k_2$](1)(3)
\tikzset{EdgeStyle/.style={post, bend right = 40,line width = 2.5}}
\Edge[label=$k_4$](3)(1)
\tikzset{EdgeStyle/.style={post,line width = 2.5}}
\Edge[label=$k_5$](3)(4)
\tikzset{EdgeStyle/.style={dashed,post,bend right = 40,line width = 2.5}}
\Edge[label={\tiny return}](4)(3)
\end{tikzpicture}
}
\hfill
\scalebox{0.3}{
\begin{tikzpicture}
 \SetUpEdge[lw         = 0.5pt,
            color      = black,
            labelstyle = {sloped,scale=2}]
  \SetVertexMath
       \tikzset{VertexStyle/.style={scale=1,
       draw,
            shape = circle,
            line width = 1pt,
            color = black,
            outer sep=1pt}}
  \Vertex[x=0,y=4]{1}
  \Vertex[x=3,y=10]{4}
  \Vertex[x=3,y=6]{3}
  \Vertex[x=6,y=4]{2}
\tikzset{EdgeStyle/.style={post, bend right = 40,line width = 2.5}}
\Edge[label=$k_2$](2)(3)
\tikzset{EdgeStyle/.style={post, bend right = 40,line width = 2.5}}
\Edge[label=$k_4$](3)(2)
\tikzset{EdgeStyle/.style={post, bend right = 40,line width = 2.5}}
\Edge[label=$k_1$](1)(2)
\tikzset{EdgeStyle/.style={post, bend right = 40,line width = 2.5}}
\Edge[label=$k_3$](3)(1)
\tikzset{EdgeStyle/.style={post,line width = 2.5}}
\Edge[label=$k_5$](3)(4)
\tikzset{EdgeStyle/.style={dashed,post,bend right = 40,line width = 2.5}}
\Edge[label={\tiny return}](4)(3)
\end{tikzpicture}
}
\hfill
\scalebox{0.3}{
\begin{tikzpicture}
 \SetUpEdge[lw         = 0.5pt,
            color      = black,
            labelstyle = {sloped,scale=2}]
  \SetVertexMath
       \tikzset{VertexStyle/.style={scale=1,
       draw,
            shape = circle,
            line width = 1pt,
            color = black,
            outer sep=1pt}}
  \Vertex[x=0,y=4]{1}
  \Vertex[x=3,y=10]{4}
  \Vertex[x=3,y=6]{3}
  \Vertex[x=6,y=4]{2}
\tikzset{EdgeStyle/.style={post, bend right = 40,line width = 2.5}}
\Edge[label=$k_1$](1)(3)
\tikzset{EdgeStyle/.style={post, bend right = 40,line width = 2.5}}
\Edge[label=$k_3$](3)(1)
\tikzset{EdgeStyle/.style={post, bend right = 40,line width = 2.5}}
\Edge[label=$k_2$](2)(3)
\tikzset{EdgeStyle/.style={post, bend right = 40,line width = 2.5}}
\Edge[label=$k_4$](3)(2)
\tikzset{EdgeStyle/.style={post,line width = 2.5}}
\Edge[label=$k_5$](3)(4)
\tikzset{EdgeStyle/.style={dashed,post,bend right = 40,line width = 2.5}}
\Edge[label={\tiny return}](4)(3)
\end{tikzpicture}
}

{ \bf $M2$ }\hfill { \bf $M3$} \hfill { \bf $M4$} 
\hfill { \bf $M8$} \hfill { \bf $M9$} \hfill

\end{center}
\caption{\label{figure2} 
 {
 Models with $N=3$ and return to the state $3$. 
 We show only those that are ergodic (every state can be reached from any other state through a directed path) and not related by symmetries (permutations of states  $1,2,3$).
All these models can generate exactly the same phase-type 
distribution.
 }
 }
\end{figure}

\section{Inverse  problem}
 {The inverse problem involves assuming that the survival function, and thus the parameters
$\lambda_1,\ldots,\lambda_N$, and $A_1,\ldots,A_{N-1}$, are known. These
can be estimated from data using least squares or maximum likelihood  \cite{dufresne2007fitting,Tantale2021,douaihy2023burstdeconv,Liuphd}. The goal is to identify which Markov
chains can represent this phase-type distribution.} A number of general results are known for this problem.
We recall here the following, important
result, that follows from  Theorem 12
in \cite{commault_phase-type_2003}:
\begin{proposition}
A multi-exponential phase-type distribution
with $N$ exponentials can be represented
by a Markov chain model 
of order $N$ (the order is defined  as the number of unobservable states). 
\end{proposition}

In this section we introduce the problem of computing the transition rate parameters, which are the non-zero, non-diagonal elements
of the matrix $\vect{Q}$ of a Markov chain of order $N$, 
from the $2N-1$ parameters of the phase-type distribution.
We are interested in models where this problem could have a unique 
solution, therefore the matrix $\vect{Q}$ has only $2N-1$ non-zero, non-diagonal elements. 
We range these elements in the $2N-1$ dimensional vector $\vect{k}$. After reordering, we place the element
$Q_{N,N+1}$ in the last position of $\vect{k}$, namely
$$
k_{2N-1} =Q_{N,N+1}.
$$

We show below that the inverse problem consists in solving $2N-1$ polynomial equations for  $\vect{k}$.

\subsection{Vieta's formulas}
The eigenvalues $\lambda_1,\ldots,\lambda_N$ are the roots of the the characteristic polynomial of  $\vect{\tilde Q}(\vect{k})$, defined as
\begin{equation}
P(\lambda) =  \mathrm{det} (  \lambda \vect{I} - \vect{\tilde Q}(\vect{k})) =  \lambda^N + a_{N-1}(\vect{k}) \lambda^{N-1} + \ldots +
a_1(\vect{k}) \lambda +  a_0(\vect{k}).
\label{secular}
\end{equation}
The coefficients $a_{i}(\vect{k})$ of the characteristic polynomial are polynomials with integer coefficients 
on the transition rates $\vect{k}$.

A first set of equations relating eigenvalues to the kinetic equations results from the Vieta's formulas
\begin{eqnarray}
L_1 &=& \sum_{i=1}^N \lambda_i  =  -a_{N-1} (\vect{k}),  \notag \\
L_2 &=& \sum_{i<j} \lambda_i \lambda_j   =  a_{N-2}  (\vect{k}), \notag  \\
& \vdots &  \notag  \\
L_N &=& \lambda_1 \lambda_2 \dots \lambda_N = (-1)^{N} a_0  (\vect{k}). \label{vieta}
\end{eqnarray}

\subsection{Eigenvector equations}\label{sec:evectors}
The amplitude parameters $A_1,A_2,\ldots,A_N$ of the survival function occur in \eqref{Ai} together with eigenvector 
components $u_{i,N}$ and solution coefficients $C_i$. We need to relate the latter to the transition rate parameters $\vect{k}$. 

First, we solve the eigenvector equation
\begin{equation}\label{eqeigen}
( \lambda \vect{I} - \vect{\tilde Q}(\vect{k})) \vect{u} =0.
\end{equation}

We look for solutions of \eqref{eqeigen} of the form
$$\vect{u}(\lambda,\vect{k}) = (u_1(\lambda,\vect{k}), \dots, u_{N}(\lambda,\vect{k})),$$ with $u_s(\lambda,\vect{k})=1$. 
In the subsection~\ref{sec:solv} we will see for which models this choice is possible. 

Because equations \eqref{eqeigen} have integer coefficients
in $\lambda,\vect{k}$, the eigenvector components $u_i(\lambda,\vect{k})$
are rational functions of $\lambda$ and $\vect{k}$.


The initial conditions satisfied by the variables $X_i$ and \eqref{solution}
provide a linear system of equations for the
constants $C_i$:
\begin{equation}
\sum_{j=1}^N u_i(\lambda_j,\vect{k}) C_j  = \delta_{i,s}, \quad 1\leq i \leq N.
\label{Csystem}
\end{equation}

Because of the non-degeneracy condition \eqref{nondegeneracy}, 
the system \eqref{Csystem} has a unique solution $C_i(\vect{\lambda},\vect{k}), \, 1\leq i \leq N$,
where $\vect{\lambda}=(\lambda_1,\ldots,\lambda_N)$.
The solutions $C_i(\vect{\lambda},\vect{k})$ are  
rational functions of $\vect{\lambda}$, and $\vect{k}$.

From  \eqref{Ai} we obtain $N-1$ independent equations for the transition rates $\vect{k}$:
\begin{equation}
-k_{2N-1} u_{N}(\lambda_i,\vect{k}) C_i(\vect{\lambda},\vect{k}) =  A_i \lambda_i, \quad 
1\leq i \leq N-1,
\label{evectors}
\end{equation}
where $k_{2N-1} = Q_{N,N+1}$.

The inverse problem is defined by the system
of $2N-1$ equations
formed by \eqref{vieta} and \eqref{evectors}.
When this system has solutions, the transition rates  $\vect{k}$ can be expressed as 
functions of the survival function parameters $\lambda_i$ and $A_i$, $1\leq i \leq N$.

\subsection{Illustrative example}
As an  example let us consider again the model $M9$ 
with a return state $s=3$, represented in Figure~\ref{figure2}.
For this model the characteristic polynomial is
$$
P(\lambda) = 
\lambda^3 + (k_1 + k_2 + k_3 + k_4 + k_5)\lambda^2 + (k_1k_2 + k_1k_4 + k_2k_3 + k_1k_5 + k_2k_5)\lambda + k_1k_2k_5.
$$

The system \eqref{eqeigen} has the solution:
$$\vect{u}(\lambda,\vect{k}) = \left( \frac{k_3}{k_1+\lambda}, \frac{k_3}{k_2+\lambda}, 1 \right).$$

The coefficients $(C_1,C_2,C_3)$ satisfy 
\begin{eqnarray}
C_1 \frac{k_3}{k_1+\lambda_1} + C_2 \frac{k_3}{k_1+\lambda_2} + 
C_3 \frac{k_3}{k_1+\lambda_3} &=& 0, \notag \\
C_1 \frac{k_3}{k_2+\lambda_1} + C_2 \frac{k_3}{k_2+\lambda_2} + 
C_3 \frac{k_3}{k_2+\lambda_3} &=& 0, \notag \\
C_1 + C_2 + C_3 &=& 1.
\end{eqnarray}
It follows that 
\begin{eqnarray}
C_1 &=& \frac{\lambda_1^2 + (k_1 + k_2)\lambda_1 +  k_1 k_2}
{(\lambda_1 - \lambda_2)(\lambda_1 - \lambda_3)}, \notag \\
C_2 &=& \frac{(k_1 + \lambda_2)(k_2 + \lambda_2)}{(\lambda_2 - \lambda_1)(\lambda_2 - \lambda_3)},\notag \\
C_3 &=& \frac{(k_1 + \lambda_3)(k_2 + \lambda_3)}{(\lambda_3 - \lambda_1)(\lambda_3 - \lambda_2)}.
\end{eqnarray}
The inverse problem for the model $M9$ consists of solving the following system
of five equations:
\begin{eqnarray}
  -(k_1 + k_2 + k_3 + k_4 + k_5)&=&L_1, \label{eq1}\\
  k_1k_2 + k_1k_4 + k_2k_3 + k_1k_5 + k_2k_5&=&L_2, \label{eq2}\\
 -k_1k_2k_5&=& L_3, \label{eq3}\\
-k_5 \frac{\lambda_1^2 + (k_1 + k_2)\lambda_1 +  k_1 k_2} 
{(\lambda_1 - \lambda_2)(\lambda_1 - \lambda_3)} &=& A_1 \lambda_1, \label{eq4}\\
-k_5 \frac{(k_1 + \lambda_2)(k_2 + \lambda_2)}{(\lambda_2 - \lambda_1)(\lambda_2 - \lambda_3)} &=& A_2 \lambda_2.
\label{eq5}
\end{eqnarray}

\subsection{Symmetrized systems}\label{sec:sym}
There is a difference between systems \eqref{vieta} and \eqref{evectors}. System \eqref{vieta} is entirely expressed using elementary symmetric polynomials in $\lambda_i$, whereas there is no obvious symmetry in system \eqref{evectors}. This difference is clearly visible in the example of model $M9$: equations \eqref{eq1}, \eqref{eq2}, and \eqref{eq3} are symmetrized, while equations \eqref{eq4} and \eqref{eq5} are not symmetric in $\lambda_i$ and $A_i$.

We show here that the system formed by \eqref{vieta} and \eqref{evectors} is equivalent to a    system  symmetrized in both $\lambda_i$ and
$A_i$. 
The advantage of a symmetrized system over a 
non-symmetrized one is that it handles 
simpler formulas, decreasing the computational 
burden of the symbolic tools. 

To this aim we use the following identities (due to Jacobi-Trudi \cite{Fulton-Harris}) that are valid for any distinct $N$ numbers $\lambda_i$ ($1\leq i \leq N$):
\begin{equation}
\sum_{i=1}^{N} \lambda_i^k \prod_{\overset{j=1}{j \neq i}}^{N} \frac{1}{\lambda_i- \lambda_j}=
\left\{
\begin{array}{ll}
0, & \text{if } k < N-1 \\
1, & \text{if } k = N-1 \\
h_{k-N+1}(\vect{\lambda}), & \text{if } k > N-1
\end{array}
\right. 
\label{jacobi}
\end{equation}
where  $h_{k-N+1}(\vect{\lambda})$ is the complete symmetric polynomial of degree $k-N+1$ in $N$ variables and $\vect{\lambda}=(\lambda_1,\ldots,\lambda_N)$. More precisely, the complete symmetric polynomials in the $N$ variables $\lambda_1,\ldots,\lambda_N$ have the form  
\begin{eqnarray}
 h_1(\vect{\lambda}) &=& \sum_{1\leq i_1 \leq N}^N \lambda_{i_1} , \ 
 h_2(\vect{\lambda}) = \sum_{1\leq i_1 \leq i_2 \leq N}^N \lambda_{i_1}\lambda_{i_2} , \ \dots , \notag \\ 
  h_m(\vect{\lambda}) &=& \sum_{1\le i_1\le \dots \le i_m\le N}\lambda_{i_1} \lambda_{i_2} \cdots \lambda_{i_m}, \ \dots  \ .   
\end{eqnarray}

In order to relate \eqref{evectors} and \eqref{jacobi} we first relate the coefficients $C_i$ to eigenvalues. 

Denote by $D_{si}(\lambda, \vect{k})$ the determinant of the $(N-1)\times (N-1)$ submatrix of the matrix $R(\lambda,\vect{k})=\lambda \vect{I}-\vect{\tilde Q}(\vect{k})$ obtained by deleting its $s$-th row and $i$-th column, for $1\le i\le N$.
Because elements of $R(\lambda,\vect{k})$ are 
linear combinations with 
integer coefficients of 
$\lambda$ and components of $\vect{k}$,
$D_{si}(\lambda, \vect{k}) \in \ZZ[\lambda, \vect{k}]$,
where $\ZZ[\lambda, \vect{k}]$ is the ring of 
polynomials with integer coefficients
in $\lambda$, and $\vect{k}$.

From the definition of $D_{si}(\lambda,\vect{k})$, it follows for $i=s$ that $\mathrm{deg}_{\lambda} (D_{ss}(\lambda,\vect{k}))=N-1$ and that the (leading) coefficient of $D_{ss}(\lambda,\vect{k})$ at the monomial $\lambda^{N-1}$ equals 1, while  for $1\le i\neq s\le N$ we have $\mathrm{deg}_{\lambda} (D_{si}(\lambda,\vect{k}))\le N-2$.
Thus, we have  
\begin{eqnarray} \label{Ds}
D_{ss}(\lambda,\vect{k})&=:&\lambda^{N-1} + c_{N-2}(\vect{k} ) \lambda^{N-2} + \ldots + c_1 (\vect{k} ) \lambda + c_0 (\vect{k} ), \\
D_{si}(\lambda,\vect{k})&=:& d_{i,N-2}(\vect{k} )\lambda^{N-2} + d_{i,N-3}(\vect{k} ) \lambda^{N-3} + \ldots + d_{i,1} (\vect{k} ) \lambda + d_{i,0} (\vect{k} ), \notag \\
&&\text{ if } i\neq s.  \label{DN}
\end{eqnarray} 

\begin{lemma}\label{kramer}
Assume that $\lambda_i\neq \lambda_j$ for $1\le i\neq j\le N$ and that $D_{ss}(\lambda_j,\vect{k}) \neq 0$ and $u_s(\lambda,\vect{k})=1$ for $ 1\le j\le N$. Then 
\begin{equation}\label{eigenvectors}
u_i(\lambda,\vect{k}) = (-1)^{|s-i|}\frac{D_{si}(\lambda,\vect{k})}{D_{ss}(\lambda,\vect{k})},
\end{equation}
and
\begin{equation}\label{90}
C_j=\frac{D_{ss}(\lambda_j,\vect{k})}{\prod_{1\le l\neq j\le N}(\lambda_j-\lambda_l)},\quad  1\le j\le N,    
\end{equation}
where $u_i(\lambda,\vect{k})$ are the eigenvector components,
and
$C_j$ form the unique solution of the system
$$\sum_{1\le j\le N} C_j=1,$$
$$\sum_{1\le j\le N} u_i(\lambda_j,\vect{k}) C_j=0, \quad 1\le i\neq s\le N.$$
\end{lemma}
\begin{proof}
If $D_{ss}(\lambda_j,\vect{k})\neq 0$, then, using Cramer's rule for the system 
\eqref{eqeigen} with $u_s(\lambda,\vect{k})=1$, we find  
$$\vect{u}(\lambda,\vect{k})=((-1)^{|s-1|}D_{s1}(\lambda,\vect{k}),\dots,(-1)^{|s-N|}D_{sN}(\lambda,\vect{k}))/D_{ss}(\lambda,\vect{k}).$$ 
Therefore, it holds
$$\sum_{1\le j\le N}  C_j=\sum_{1\le j\le N} \frac{D_{ss}(\lambda_j,\vect{k})}{\prod_{1\le l\neq j\le N}(\lambda_j-\lambda_l)}=1,$$ 
\noindent while 
$$\sum_{1\le j\le N} u_i(\lambda_j,\vect{k}) C_j=\sum_{1\le j\le N} \frac{(-1)^{|s-i|} D_{si}(\lambda_j,\vect{k})}{\prod_{1\le l\neq j\le N}(\lambda_j-\lambda_l)}=0, \quad  1\le i\neq s\le N,$$ 
\noindent where $C_j, 1\le j\le N$ are taken from (\ref{90})) due to (\ref{jacobi}) and to the 
fact that the degree in $\lambda$ of $D_{si}(\lambda,\vect{k})$ is $N-1$ and is smaller than $N-1$ for $i=s$
and $i\neq s$, respectively.
\end{proof}




The structure of the equations \eqref{evectors} suggests a natural way to symmetrize them. 
For $1\leq k \leq N-1$ we define $S_k$ to be the sum of $A_1 \lambda_1^k,\dots, A_N \lambda_N^k$. Using \eqref{evectors},\eqref{eigenvectors}  and \eqref{90} we obtain for this sum the formula
\begin{equation}\label{Sk}
S_k = \sum_{i=1}^N A_i \lambda_i^k =   
k_{2N-1}  \sum_{i=1}^N \lambda_i^{k-1} \frac{(-1)^{|s-N|+1}D_{sN}(\lambda_i,\vect{k})}{\prod_{1\le j\neq i\le N}(\lambda_i-\lambda_j)}.
\end{equation}


We can distinguish two cases:

\paragraph{i) $\mathbf s=N$.} 
In this case, according to \eqref{Ds}  $D_{sN}$ has degree $N-1$ in $\lambda$.
Using \eqref{Sk} 
and the Jacobi-Trudi identities \eqref{jacobi} we find the following $N-1$ symmetrized equations:
\begin{eqnarray}\label{eq:sym}
 S_1 &=& \sum_{i=1}^N A_i \lambda_i  = - k_{2N-1},   \notag \\
 S_2 &=& \sum_{i=1}^N A_i \lambda_i^2  = - k_{2N-1} ( h_1(\vect{\lambda}) + c_{N-2}(\vect{k} )),   \notag \\
 &\vdots& \notag \\
 S_{N-1} &=& \sum_{i=1}^N A_i \lambda_i^{N-1}  =  
 - k_{2N-1} (  h_{N-2}(\vect{\lambda}) + c_{N-2}(\vect{k}) h_{N-3}(\vect{\lambda}) + \ldots +
c_{2}(\vect{k})  h_1(\vect{\lambda}) + c_{1}(\vect{k} )). \notag \\
\end{eqnarray}
The inverse problem is thus equivalent to solving the symmetrized equations in \eqref{vieta} and \eqref{eq:sym} and 
computing the transition rates 
$\vect{k}$ as functions of the symmetric polynomials $L_1,\ldots,L_N$ and $S_1,\ldots,S_{N-1}$. We should note that
the complete symmetric polynomials $h_1(\vect{\lambda}),\ldots,  h_{N-2}(\vect{\lambda})$ can be expressed 
using the elementary symmetric polynomials $L_1,\ldots, L_{N-2}$, that is $h_1(\vect{\lambda}) = L_1$, $h_2(\vect{\lambda}) = L_1^2 - L_2$ and so on.

\paragraph{ii) $\mathbf s < N$.} 
In this case, according to \eqref{DN}  $D_{sN}$ has degree at most $N-2$. Using \eqref{Sk} and \eqref{jacobi} it follows that 
\begin{equation} \label{constraint}
 S_1 = \sum_{i=1}^N A_i \lambda_i = 0.
 \end{equation}
This means that the inverse problem is not well posed in this case. The $2N-1$ parameters 
$\lambda_1,\ldots,\lambda_N$, $A_1,\ldots,A_{N-1}$ are no longer independent, they need
to satisfy the constraint \eqref{constraint}.
Furthermore, we can no longer determine all the transition rate parameters
$\vect{k}$ uniquely.
There remain $2N-2$ equations that can be used to constrain the $2N-1$ transition rate parameters
\begin{eqnarray}\label{eq:sym2}
 S_2 &=& (-1)^{|s-N|+1} k_{2N-1} d_{N,N-2}(\vect{k}),   \notag \\
S_3 &=& (-1)^{|s-N|+1} k_{2N-1}(d_{N,N-2}(\vect{k})h_1(\vect{\lambda}) + d_{N,N-3}(\vect{k}) ),    \notag \\
 &\vdots& \notag \\
S_k &=& (-1)^{|s-N|+1} k_{2N-1}(d_{N,N-2}(\vect{k})h_{k-2}(\vect{\lambda}) + d_{N,N-3}(\vect{k})h_{k-3}(\vect{\lambda}) + \ldots +
d_{N,N-k}(\vect{k}) ) ,   \notag \\
  &\vdots& \notag \\
 S_{N-1} &=& 
 (-1)^{|s-N|+1} k_{2N-1} ( d_{N,N-2}(\vect{k})  h_{N-3}(\vect{\lambda}) + \ldots +
d_{N,2}(\vect{k})  h_1(\vect{\lambda}) + d_{N,1}(\vect{k} )).  
\end{eqnarray}
In this case the symmetrized system made of \eqref{vieta},\eqref{eq:sym2} can be used to compute the transition rates 
$\vect{k}$ as functions of the symmetric polynomials 
$L_1,\ldots,L_N$ and $S_2,\ldots,S_{N-1}$ and of one or several indeterminate transition rates. 

 {While the examples in this paper assume $N=s$,
the case $N\neq s$
is also relevant in practice. 
In this case, additional data such as state occupancy probabilities and lifetimes help to constrain the kinetic parameters by eliminating the remaining indeterminate variables. Symmetrized equations provide further constraints that must be solved alongside those from the supplementary data.
}

\subsection{Symmetrized system for the illustrative example}

For the illustrative example of the model $M9$, the matrix $R(\lambda,\vect{k})$ reads
$$R(\lambda,\vect{k})=
\begin{pmatrix}
\lambda +k_1 &  0 & -k_3  \\
0 & \lambda +k_2 & -k_4  \\
-k_1 & -k_2 & \lambda +(k_3+k_4+k_5)
\end{pmatrix}.
$$
Considering that the return state is $s=3$ we have 
$$
D_{33} (\lambda,\vect{k}) = \det 
\begin{pmatrix}
\lambda +k_1 & 0   \\
0  & \lambda + k_2 
\end{pmatrix} = (\lambda +k_1) (\lambda + k_2), 
$$
and the symmetrized equations 
$$
S_1 = - k_5,
$$
$$
S_2 = - k_5 ( L_1 +  k_1 + k_2  ).
$$
If the return state is $s=1$, we have
$$
D_{33} (\lambda,\vect{k}) = \det 
\begin{pmatrix}
0  & \lambda +k_2    \\
-k_1  & - k_2 
\end{pmatrix} = k_1 (\lambda + k_2), 
$$
and the symmetrized equations 
$$
S_1 = 0,
$$
$$
S_2 = - k_5 ( L_1 +  k_1   ).
$$
Finally, if the return state is $s=2$, we have
$$
D_{33} (\lambda,\vect{k}) = \det 
\begin{pmatrix}
 \lambda +k_1 & 0    \\
-k_1  & - k_2 
\end{pmatrix} = -k_2 (\lambda + k_1),
$$
and the symmetrized equations 
$$
S_1 = 0,
$$
$$
S_2 =  k_5 ( L_1 -  k_2   ).
$$

\begin{figure}[h!]
\begin{center}
\scalebox{0.6}{
\begin{tikzpicture}
 \SetUpEdge[lw         = 0.5pt,
            color      = black,
            labelstyle = {sloped,scale=2}]
  \SetVertexMath
       \tikzset{VertexStyle/.style={scale=1,
       draw,
            shape = circle,
            line width = 1pt,
            color = black,
            outer sep=1pt}}
  \Vertex[x=0,y=4]{1}
  \Vertex[x=3,y=10]{4}
  \Vertex[x=3,y=6]{3}
  \Vertex[x=6,y=4]{2}
\tikzset{EdgeStyle/.style={post, bend right = 20,line width = 2.5}}
\Edge[label=$k_1$](1)(2)
\tikzset{EdgeStyle/.style={post, bend left = 20,line width = 2.5}}
\Edge[label=$k_3$](3)(2)
\tikzset{EdgeStyle/.style={post, bend  right = 20,line width = 2.5}}
\Edge[label=$k_2$](1)(3)
\tikzset{EdgeStyle/.style={post, bend right = 20,line width = 2.5}}
\Edge[label=$k_4$](3)(1)
\tikzset{EdgeStyle/.style={post,line width = 2.5}}
\Edge[label=$k_5$](3)(4)
\tikzset{EdgeStyle/.style={dashed,post,bend right = 40,line width = 2.5}}
\Edge[label={\tiny return}](4)(3)
\end{tikzpicture}
}
\hskip2truecm
\scalebox{0.6}{
\begin{tikzpicture}
 \SetUpEdge[lw         = 0.5pt,
            color      = black,
            labelstyle = {sloped,scale=2}]
  \SetVertexMath
       \tikzset{VertexStyle/.style={scale=1,
       draw,
            shape = circle,
            line width = 1pt,
            color = black,
            outer sep=1pt}}
  \Vertex[x=0,y=4]{1}
  \Vertex[x=3,y=10]{4}
  \Vertex[x=3,y=6]{3}
  \Vertex[x=6,y=4]{2}
\tikzset{EdgeStyle/.style={post, bend left = 20,line width = 2.5}}
\Edge[label=$k_1$](2)(1)
\tikzset{EdgeStyle/.style={post, bend right = 20,line width = 2.5}}
\Edge[label=$k_3$](2)(3)
\tikzset{EdgeStyle/.style={post, bend  right = 20,line width = 2.5}}
\Edge[label=$k_2$](1)(3)
\tikzset{EdgeStyle/.style={post, bend right = 20,line width = 2.5}}
\Edge[label=$k_4$](3)(1)
\tikzset{EdgeStyle/.style={post,line width = 2.5}}
\Edge[label=$k_5$](3)(4)
\tikzset{EdgeStyle/.style={dashed,post,bend right = 40,line width = 2.5}}
\Edge[label={\tiny return}](4)(3)
\end{tikzpicture}
}

 { \bf a) }\hskip6truecm { \bf b)}  

\end{center}
\caption{\label{figure3} 
Non ergodic models with return. In both examples
the state $2$ is not in the strongly connected component of the 
return state $3$. 
a) the state $2$ is absorbing,  
 {so} the return time can be infinite 
with non-zero probability,  {meaning that} the distribution is not of the phase-type. 
b) the state $2$ cannot be reached 
from the return state $3$, 
  {allowing the ergodic chain obtained by eliminating state}
the state $2$  {to} be used to compute
the phase-type distribution.
 }
\end{figure}

\subsection{Solvable models}\label{sec:solv}
Different Markov chain models are distinguished by their transition graph  defined as the directed graph $G=(V,A)$ with vertices $V = \{1,\ldots,N+1\}$
and arcs $A = \{(i,j) \mid i,j\in V,\, i \neq j,\,  Q_{ij} \neq 0\}$. 

Let us consider that this graph satisfies some conditions related to the underlying modeled process:
\begin{enumerate}
\item[(C1)]
$N+1$ is reachable only from $N$. 
\item[(C2)]
From $N+1$
there is only one return arc, towards $s$, where $1\leq s \leq N$.
\end{enumerate}
C1 is a simplifying assumption that is satisfied in a number of applications, including the biological application discussed in this paper.  
 {All the results in this section remain valid, with slight modifications, even without this condition (see Subsection~\ref{sec:noC1}).
}
The condition C2 is needed in order to have a phase-type distribution.

For a directed graph $G$ and its vertices $v,w$ we denote $v\preceq w$ if there is a path in $G$ from $v$ to $w$. We say that $v,w$ are equivalent  {if} $v\preceq w, w\preceq v$. Equivalence classes of nodes are called strongly connected components (SCC).
If $G$ consists of a single SCC then we call $G$ {\em strongly connected}. 
A Markov chain model is called {\em   ergodic} if and only if its transition graph is strongly connected,
 {meaning that every node can be reached from any other node through a directed path.}
Let us consider the SCC of the return state $s$,
 {i.e. the smallest
strongly connected sub-graph containing $s$}.
Non ergodic models contain nodes $w$ outside this SCC. Two situations may arise for such nodes, as represented in Figure~\ref{figure3}. In Figure~\ref{figure3}a), $s \preceq w$ but the converse is not true. With non-zero probability, the chain starting
in $s$ never returns to $s$. This situation does not lead to a phase-type distribution and will not be considered. In Figure~\ref{figure3}b), $w$ can not be reached from $s$. In such a case, an ergodic model with fewer states, limited to the SCC of $s$, generates the phase-type distribution. 
For these reasons we can restrict our analysis to ergodic models only. 
We would like to know for which ergodic models  the inverse problem has a unique solution, eventually up to transformation by discrete symmetries. 
\begin{definition}\label{def:solvable}
We say that an model is {\em solvable} if the following solutions are satisfied:
\begin{enumerate}[label=(\roman*)]
\item 
The transition graph is strongly  {connected.}
\label{def30}
\item  \label{def3ii}
The model has $2N-1$ transition rate parameters (non-zero, non-diagonal elements of $\vect{\tilde Q}$).
\item  \label{def3iii}
The system made by the equations \eqref{vieta} and \eqref{evectors} or equivalently the system
made by the equations \eqref{vieta} and \eqref{eq:sym} or the system
\eqref{vieta},\eqref{eq:sym2}
has unique solutions up to transformations
by discrete symmetries, on a open domain of dimension $2N-1$. 
\end{enumerate}
\end{definition}
It is very difficult to obtain general sufficient conditions for solvability but we can state a necessary condition.  
\begin{proposition}\label{prop:p4}
A solvable model  necessarily satisfies $s=N$.  
\end{proposition}
 {\begin{proof}
If $s\neq N$, then \eqref{constraint} holds, making the survival function parameters independent. The
inverse problem then involves at most $2N-2$ independent
equations for the $2N-1$ kinetic
parameters. 
\end{proof}
\begin{remark}
In Subsection~\ref{sec:noC1} we relax this necessary condition  to  $s \in \text{Pred}(N+1)$, where $\text{Pred}(N+1)$ is the set of predecessors of 
$N+1$ in the transition graph. 
\end{remark}
}

Ergodic models also satisfy the following property that has been used in condition \eqref{nondegeneracy} and 
is needed for writing the solution \eqref{solution}.
\begin{proposition}\label{nonzero_eigenvalues}
All ergodic models satisfy $\max_{1 \leq i \leq N}\lambda_i < 0$, where
$\lambda_i$ are the eigenvalues of $\vect{\tilde Q}(\vect{k})$.
\end{proposition}
\begin{proof}
Suppose that some $\lambda_i=0$. Then 
we have 
$$
\vect{\tilde Q}(\vect{k}) \vect{u} (\lambda_i, \vect{k}) = 0.
$$
Summing all the equations of the above system term by term, we obtain
$
k_{2N-1} u_N (\lambda_i, \vect{k}) = 0.
$
This implies that $u_N (\lambda_i, \vect{k}) = 0$.
Therefore,
$\vect{u} (\lambda_i, \vect{k})$ also satisfies
$$
\vect{\tilde Q}^{\mathrm{red}} \vect{u} (\lambda_i, \vect{k}) = 0,
$$
where $\vect{\tilde Q}^{\mathrm{red}}$ is obtained from  
$\vect{\tilde Q}(\vect{k})$ by setting $k_{2N-1}=0$. 

Note that $\vect{\tilde Q}^{\mathrm{red}}$ satisfies the zero column-wise sum 
and is the dual generator of a reduced Markov chain obtained
by deleting the state $N+1$.
Therefore $\vect{u} (\lambda_i, \vect{k})/\sum_{i=1}^N u_i (\lambda_i, \vect{k})$ is a steady state probability distribution
of the reduced Markov chain with states $\{1,\ldots,N\}$ 
and dual generator $\vect{\tilde Q}^{\mathrm{red}}$.
 
The full model, and consequently the
reduced chain, are  ergodic. Or,  any  ergodic Markov chain has a unique 
steady state distribution in which all states have non-zero probabilities, which 
contradicts $u_N (\lambda_i, \vect{k}) = 0$.
\end{proof}

The following Proposition guarantees 
that the system of equations 
\eqref{evectors}, that are used to
define the 
inverse problem, can be obtained
for all ergodic models.
\begin{proposition}
Consider that $G$ is strongly connected.  Then
the equation $( \lambda \vect{I} - \vect{\tilde Q}(\vect{k})) \vect{u} =0$
has solutions
$\vect{u}(\lambda,\vect{k}) = (u_1(\lambda,\vect{k}), \dots, u_{N}(\lambda,\vect{k}))$ with $u_i(\lambda,\vect{k}) \neq 0$
for all $1\leq i \leq N$. 
\end{proposition}
\begin{proof}
With the notations of the  Subsection~\ref{sec:evectors}, the polynomial
$D_{ss}(\lambda, \vect{k})$ has leading term $\lambda^{N-1}$ and therefore is not identically zero (see \eqref{Ds}).   Lemma~\ref{kramer} implies that there are eigenvector solutions with $u_s(\lambda, \vect{k}) = 1$.  
Furthermore, for  $i\neq s$,  
$u_i(\lambda, \vect{k})$ are not identically zero
because, otherwise, $u_i(0, \vect{k})$ would also be zero. 
However, $\vect{u}(0, \vect{k})$ 
is an eigenvector of
$\vect{\tilde Q}(\vect{k})$ for the
eigenvalue $\lambda=0$, representing a steady state
of the Markov chain. Following the same argument as in the proof of  Proposition~\ref{nonzero_eigenvalues}, for ergodic chains,  
 $u_i(0, \vect{k}) \neq 0$ for all $1\leq i \leq N$.
\end{proof}


Figure~\ref{figure2} shows all the ergodic models with $N=3$ satisfying the conditions C1 and C2 and having $2N-1$ non-zero transition rate parameters. The numbering of the models stems from the fact that there are nine models that satisfy the conditions, and among them, only five are ergodic.
 
 {\subsection{Symmetrized equations and solvability without the condition C1.}\label{sec:noC1}
In this subsection, we lift the simplifying assumption that $N+1$
is reachable only from $N$ and
allow multiple states to lead to $N$.
The $2N-1$ independent kinetic parameters of the model are non-zero, non-diagonal elements of $\vect{Q}$.

Let $\text{Pred}(N+1)$ be the set of states that lead to $N+1$ (the predecessors of $N+1$ in the transition graph).

Then, \eqref{eq:probaevolutionN},
\eqref{Ai}, and \eqref{Sk}
become
$$\D{X_{N+1}}{t} = \sum_{j \in \text{Pred}(N+1)} Q_{j,N+1} X_j {(t)},$$
$$
 A_i = - \frac{C_i}{\lambda_i} \sum_{j \in \text{Pred}(N+1)}  Q_{j,N+1}  u_{i,j},
$$
$$
S_k =  \sum_{i=1}^N
 \frac{
 \sum_{j \in \text{Pred}(N+1)} (-1)^{|s-j|+1} Q_{j,N+1}  \lambda_i^{k-1}
 D_{sj}(\lambda_i,\vect{Q})
}{\prod_{1\le l\neq i\le N}(\lambda_i-\lambda_l)},
$$
where  $D_{sj}$ are defined as in Subsection~\ref{sec:sym}.

The Vieta's formulas  \eqref{vieta}, which provide the first $N$ equations of the inverse problem, remain unchanged. 
The remaining $N-1$ symmetrized equations resulting from the eigenvectors are now more complex.  

Like in Subsection~\ref{sec:sym}, we have two cases: 

\paragraph{
i) $s \notin \text{Pred}(N+1)$.}

In this case $S_1 = 0$, and the argument in Proposition~\ref{prop:p4} 
holds: the model is not solvable. 
The remaining $N-2$ symmetrized equations resulting from eigenvectors read 
\begin{eqnarray}\label{eq:sym2noC1}
 S_2 &=& 
 \sum_{j\in \text{Pred}(N+1)}
 (-1)^{|s-j|+1} Q_{j,N+1} d_{N,N-2}(\vect{Q}),   \notag \\
S_3 &=& 
\sum_{j\in \text{Pred}(N+1)}
(-1)^{|s-j|+1} Q_{j,N+1}(d_{N,N-2}(\vect{Q})h_1(\vect{\lambda}) + 
d_{N,N-3}(\vect{Q}) ),    \notag \\
 &\vdots& \notag \\
S_k &=& 
\sum_{j\in \text{Pred}(N+1)}
(-1)^{|s-j|+1} Q_{j,N+1}(d_{N,N-2}(\vect{Q})h_{k-2}(\vect{\lambda}) + d_{N,N-3}(\vect{Q})h_{k-3}(\vect{\lambda}) +\notag \\
&+&\ldots +
d_{N,N-k}(\vect{Q}) ) ,   \notag \\
  &\vdots& \notag \\
 S_{N-1} &=& 
 \sum_{j\in \text{Pred}(N+1)}
 (-1)^{|s-j|+1} Q_{j,N+1} ( d_{N,N-2}(\vect{Q})  h_{N-3}(\vect{\lambda}) + \ldots +
d_{N,2}(\vect{Q})  h_1(\vect{\lambda}) +\notag \\ 
&+&d_{N,1}(\vect{Q} )),  
\end{eqnarray}
where $d_{N,k}(\vect{Q}),\,0\leq  k \leq N-2$ 
are defined as in 
Subsection~\ref{sec:sym}.

\paragraph{
ii) $s \in \text{Pred}(N+1)$.}

In this case we have $N-1$ equations
resulting from the eigenvectors:
\begin{eqnarray}\label{eq:symnoC1}
 S_1 &=&  - Q_{s,N+1},   \notag \\
 S_2 &=&  - Q_{s,N+1} ( h_1(\vect{\lambda}) + c_{N-2}(\vect{Q} )) +\notag \\
&+& \sum_{j \in \text{Pred}(N+1), j\neq s}  (-1)^{|s-j|+1}Q_{j,N+1} d_{j,N-2}(\vect{Q}),   \notag \\
 &\vdots& \notag \\
 S_{N-1} &=&  
 - Q_{s,N+1} (  h_{N-2}(\vect{\lambda}) + c_{N-2}(\vect{Q}) h_{N-3}(\vect{\lambda}) + \ldots +
c_{2}(\vect{Q})  h_1(\vect{\lambda}) + c_{1}(\vect{Q} )) + \notag \\
&+&\sum_{j\in \text{Pred}(N+1),j\neq s}
(-1)^{|s-j|+1}Q_{j,N+1}
(d_{j,N-2}(\vect{Q})h_{N-3}(\vect{\lambda})+ 
d_{j,N-3}(\vect{Q})h_{N-4}(\vect{\lambda})+ \notag \\
&+&\ldots+d_{j,1}(\vect{Q})),  
\end{eqnarray}
where $c_{k}(\vect{Q}),\,0\leq  k \leq N-2$ 
are defined as in 
Subsection~\ref{sec:sym}.
}

\section{Solution of the inverse problem for the unbranched chain model}\label{vier}



\begin{figure}[h!]
\begin{center}
\scalebox{0.5}{
\begin{tikzpicture}
 \SetUpEdge[lw         = 0.5pt,
            color      = black,
            labelstyle = {sloped,scale=2}]
  \SetVertexMath
       \tikzset{VertexStyle/.style={scale=1, minimum size =8ex, 
       draw,
            shape = circle,
            line width = 1pt,
            color = black,
            outer sep=1pt}}
  \Vertex[x=0,y=2]{1}
  \Vertex[x=5,y=2]{2}
  \Vertex[x=10,y=2]{3}
  \Vertex[x=15,y=2]{N-1}
  \Vertex[x=20,y=2]{N}
  \Vertex[x=25,y=2]{N+1}
\tikzset{EdgeStyle/.style={post, bend right = 20,line width = 2.5}}
\Edge[label=$k_1^+$](1)(2)
\Edge[label=$k_1^-$](2)(1)
\Edge[label=$k_2^+$](2)(3)
\Edge[label=$k_2^-$](3)(2)
\Edge[label=$k_{N-1}^+$](N-1)(N)
\Edge[label=$k_{N-1}^-$](N)(N-1)
\Edge[label=$\dots$](3)(N-1)
\Edge[label=$\dots$](N-1)(3)
\tikzset{EdgeStyle/.style={post,line width = 2.5}}
\Edge[label=$k_N$](N)(N+1)
\tikzset{EdgeStyle/.style={dashed,post,bend right = 40,line width = 2.5}}
\Edge[label={\tiny return}](N+1)(N)
\end{tikzpicture}
}
\end{center}
\caption{\label{figure4} 
A solvable model 
 {with an arbitrarily large number of states}
: the unbranched chain model. For this model, the solution of the inverse problem can be computed symbolically for any number of states $N$. 
 }
\end{figure}

In this section, we present an example of a solvable model with an arbitrary number of states and propose a method to compute the solutions of the inverse problem.

Consider now a Markov chain with $N+1$ states arranged in a line and reversibly connected, as illustrated in Figure~\ref{figure4}. For this model 
$${\tilde Q}_{i+1,i}=k_i^+, {\tilde Q}_{i,i+1}=k_i^-, 1\le i<N, {\tilde Q}_{i,i}=-k_i^+ -k_{i-1}^-, 1\le i<N, {\tilde Q}_{N,N}=-k_N-k_{N-1}^-.$$
This model is a generalization, of arbitrary length $N$, of the model $M2$ represented in Figure~\ref{figure2} that has $N=3$.
To ensure that the model is solvable, as stated in  Proposition~\ref{prop:p4}, we consider the return state to be 
$s=N$.

The solution of the direct problem presented in the Section~\ref{direct_problem}
provides an algebraic map\footnote{although the rate parameters 
are real positive numbers, for formal reasons we define this map on complex numbers. 
} 
\begin{equation}\label{91}
f: \CC^{2N-1} \to \CC^{2N-1},\, f(k_1^+,\dots, k_{N-1}^+,k_1^-,\dots,k_{N-1}^-,k_N)=(A_1,\dots,A_{N-1},\lambda_1,\dots,\lambda_N).   
\end{equation}
The goal of this section is to prove that $f$ is invertible and that its inverse is a rational map. This also means that the unbranched chain model is solvable for any $N$. 
However, as we will see later, 
the inverse map is not a rational function 
for all solvable models; for instance inversion formulas may involve radicals. From this perspective, the unbranched chain model is special and, in some ways, simpler. 

Let $C:=(C_1,\dots,C_N)^T$ represent a vector,
where $C_i$ are solutions of \eqref{Csystem}. 

Note that \eqref{Csystem}, \eqref{eq:sym} imply that 
\begin{eqnarray}
\vect{U} C &=& (0,\ldots,0,1)^T ,\label{92bis} \\
k_N&=&-\sum_{1\le j\le N} A_j \lambda_j,
\label{92}
\end{eqnarray}
where $\vect{U}$  is the  $N\times N$ matrix with the columns $\vect{u}(\lambda_1,\vect{k}),\dots, \vect{u}(\lambda_N,\vect{k})$ and
$\lambda_1,\ldots,\lambda_N$ are the
eigenvalues of $\vect{\tilde Q}(\vect{k})$.

The following lemmas will provide an algorithmic approach to constructing the solution to the inverse problem.
\begin{lemma}\label{plus}
For a suitable rational function $G_{N-i}$ with rational coefficients and with a denominator $k_{N-i+1}^+\cdots k_{N-1}^+ k_{N-i}^-\cdots k_{N-1}^- k_N$ it holds
$$k_{N-i}^+=G_{N-i}(A_1,\dots,A_{N-1}, k_{N-i+1}^+,\dots, k_{N-1}^+ k_{N-i}^-,\dots, k_{N-1}^-, k_N, \lambda_1,\dots, \lambda_N),$$
for $1 \leq i \leq  N-1$.
\end{lemma}

\begin{proof}
For each eigenvalue $\lambda$ of the matrix 
$\bf \tilde Q(\vect{k})$ it holds
$$k_{N-1}^+u_{N-1}(\lambda,\vect{k})=k_N+k_{N-1}^-+\lambda,\, k_n^+u_n(\lambda,\vect{k})=(k_{n+1}^++k_n^-+\lambda)u_{n+1}(\lambda,\vect{k})-k_{n+1}^-u_{n+2}(\lambda,\vect{k}), $$
for $ 0\le n\le N-2$, considering that
$u_0(\lambda,\vect{k}) = 0$.
\noindent This provides by recursion on $N-n$ a rational function $g_n$ with rational coefficients such that
\begin{equation}\label{98}
k_n^+u_n(\lambda,\vect{k})=g_n(k_{n+1}^+,\dots, k_{N-1}^+, k_n^-,\dots, k_{N-1}^-, k_N, \lambda).    
\end{equation}
Moreover, the denominator of $g_n$ equals $k_{n+1}^+\cdots k_{N-1}^+$.

 Then for $i\ge 0$ it holds
\begin{equation}\label{99}
{\bf \tilde Q}^i UC=U\cdot \mathrm{diag} (\lambda_1^i,\dots, \lambda_N^i)C= {\bf \tilde Q}^i (0,\dots, 0, 1)^T,
\end{equation}
where $\mathrm{diag} (\lambda_1^i, \dots, \lambda_N^i)$ denotes a diagonal matrix.

The $(N-i)$-th coordinate of the middle vector in (\ref{99}) equals
\begin{equation}\label{97}
\sum_{1\le j\le N} u_{N-i} (\lambda_j,\vect{k})\lambda_j^i (-A_j\lambda_j/k_N).    
\end{equation}
The same $(N-i)$-th coordinate of the right vector in (\ref{99}) equals
\begin{equation}\label{96}
 k_{N-i}^- \cdots k_{N-1}^-.   
\end{equation}
Observe that the $j$-th coordinate of the right vector in (\ref{99}) vanishes for $1\le j<N-i$.

Multiplying  both sides of (\ref{98}) for $n=N-i, \lambda=\lambda_j$ by $\lambda_j^i(-A_j\lambda_j/k_N)$ and summing them up over $1\le j\le N$, we obtain that
\begin{eqnarray}\label{95}
k_{N-i}^+k_{N-i}^-\cdots k_{N-1}^-&=& G_{N-i,0} (A_1,\dots, A_{N-1},k_{N-i+1}^+,\dots, k_{N-1}^+, k_{N-i}^-,
\notag \\
&\dots&,
k_{N-1}^-, k_N, \lambda_1,\dots, \lambda_N)
\end{eqnarray}
for a suitable rational function $G_{N-i,0}$ with rational coefficients and with a denominator $$k_{N-i+1}^+\cdots k_{N-1}^+ k_N$$ taking into account the equality of (\ref{97}) and of (\ref{96}). 
\end{proof}



\begin{lemma}\label{minus}
For an appropriate rational function $P_{N-i}$ with rational coefficients and with a denominator $k_{N-i+1}^+ \cdots k_{N-1}^+ k_{N-i+1}^- \cdots k_{N-1}^- k_N$ it holds
$$k_{N-i}^-=P_{N-i}(A_1,\dots,A_{N-1}, k_{N-i+1}^+, \dots, k_{N-1}^+, k_{N-i+1}^-, \dots, k_{N-1}^-, k_N, \lambda_1,\dots, \lambda_N).$$
\end{lemma}

\begin{proof} The $(N-i+1)$-th coordinate of the middle vector in (\ref{99}) equals
\begin{equation}\label{94}
\sum_{1\le j\le N} u_{N-i+1} (\lambda_j,\vect{k}) \lambda_j^i (-A_j\lambda_j/k_N).    
\end{equation}
The same $(N-i+1)$-th coordinate of the right vector in (\ref{99}) equals
\begin{equation}\label{93}
-k_{N-i}^-\cdots k_{N-1}^- + p_{N-i}(k_{N-i+1}^+,\dots, k_{N-1}^+, k_{N-i+1}^-,\dots, k_{N-1}^-,k_N) \end{equation}
for an appropriate polynomial $p_{N-i}$ with integer coefficients.

Multiplying both sides of (\ref{98}) for $n=N-i+1, \lambda=\lambda_j$ by $\lambda_j^i (-A_j\lambda_j/k_N)$ and summing them up over $1\le j\le N$, we obtain that
$$-k_{N-i+1}^+k_{N-i}^- \cdots k_{N-1}^- +  h_{N-i}(k_{N-i+1}^+,\dots, k_{N-1}^+, k_{N-i+1}^-,\dots, k_{N-1}^-,k_N) =$$ $$P_{N-i,0} (A_1,\dots, A_{N-1}, k_{N-i+2}^+,\dots, k_{N-1}^+, k_{N-i+1}^-,\dots, k_{N-1}^-, k_N, \lambda_1,\dots, \lambda_N)$$
\noindent for a suitable rational function $P_{N-i,0}$ with rational coefficients and with a denominator $k_{N-i+2}^+ \cdots k_{N-1}^+ k_N$ taking into account the equality of (\ref{94}) and of (\ref{93}) (while the both latter multiplied by $k_{N-i+1}^+$). 
\end{proof}



Applying alternatingly Lemma~\ref{minus} and Lemma~\ref{plus} for $i=1,\dots, N-1$ consecutively, we conclude with the following main result of this section.

\begin{theorem}\label{interval}
There is an algorithm which produces rational functions $F_n^+, F_n^-, 1\le n\le N-1$ with rational coefficients such that 
$$k_n^+=F_n^+(A_1,\dots,A_{N-1}, \lambda_1,\dots, \lambda_N), k_n^-=F_n^-(A_1,\dots,A_{N-1}, \lambda_1,\dots, \lambda_N).$$
\end{theorem}

\begin{remark}
\begin{enumerate}[label=(\roman*)]
    \item Together with (\ref{92}) Theorem~\ref{interval} assures the inverse rational map to $f$ (see (\ref{91})) on the open dense subset of $\CC^{2N-1}$ determined by conditions $k_1^+\cdots k_{N-1}^+ \cdot k_1^-\cdots k_{N-1}^- \cdot k_N\neq 0$ and $\lambda_i \neq \lambda_j, 1\le i< j\le N$;
\item the proof of Theorem~\ref{interval} provides an algorithm which produces explicitly rational functions $F_n^+, F_n^-$ by recursion on $N-n$. 
\end{enumerate}
\end{remark}

\begin{remark}
In fact, for any model (not  {necessarily}, the unbranched chain model elaborated in this section) one can consider  {a rational map
similar to the one defined by} (\ref{91}) . Thus, $A_1,\dots,A_{N-1},\lambda_1,\dots,\lambda_{N-1}$ are rational functions in $k_n^+, k_n^-, 1\le n<N, k_N$. Therefore, there is a field extension
$$\CC(A_1,\dots,A_{N-1},\lambda_1,\dots,\lambda_{N-1}) \subset \CC(k_1^+,\dots,k_{n-1}^+, k_1^-,\dots, k_{n-1}^-, k_N)$$
\noindent where $\CC(A_1,\dots,A_{N-1},\lambda_1,\dots,\lambda_{N-1})$ denotes the field generated by $A_1,\dots,A_{N-1},\lambda_1,\dots,\lambda_{N-1}$ over the field of complex numbers. It is known (see e.g. \cite{Shafarevich}, Ch. 1) that the degree of this extension equals the number of solutions in $\CC^{2N-1}$ of the system of rational equations $A_1=\alpha_1,\dots, A_{N-1}=\alpha_{N-1}, \lambda_1=\beta_1,\dots, \lambda_N=\beta_N$ at a generic point $(\alpha_1,\dots,\alpha_{N-1},\beta_1,\dots,\beta_N)\in \CC^{2N-1}$.  Recall that the degree is defined as the dimension of the vector space $\CC(k_1^+,\dots,k_{n-1}^+, k_1^-,\dots, k_{n-1}^-, k_N)$ over the field $\CC(A_1,\dots,A_{N-1},\lambda_1,\dots,\lambda_{N-1})$. The degree can be infinite. Theorem~\ref{interval} states that for the unbranched chain model the degree equals 1. 
This provides a rigorous algebraic reformulation of the statement that 
the inverse problem is in some ways simpler for the unbranched chain model.

We conjecture that for all other models the degree is greater than 1.
\end{remark}

\section{Using the Thomas decomposition for solving the inverse problem}

Thomas decomposition is a computer algebra algorithm that decomposes systems of polynomial equations and inequations into simpler systems. These can be more easily tackled by iteratively  solving univariate polynomial equations. We apply this technique to compute the solutions of the inverse problem. The method is illustrated by solving the inverse problem for all ergodic models with 
$N=3$ but can also be applied to larger 
$N$. Additionally, the method helps determine whether a model is solvable.


\subsection{The Thomas decomposition of an algebraic system}
In this section we introduce briefly the notion of the algebraic Thomas decomposition (see the appendix of \cite{LANGEHEGERMANN2021102266} for a similar introduction). We will then apply the algebraic Thomas decomposition in the subsequent sections to our symmetrized systems to determine their solutions. 

Let $\CC [\bbx] $ be a polynomial ring in $n$ variables $\bbx =(x_1,\dots,x_n)$ over the complex numbers $\CC$. An algebraic system $\algsystem$ is defined as a finite set of polynomial equations and inequations, that is as the set  
\begin{equation}\label{eqn:algebraicsystem}
\algsystem = \{ p_1(\bbx )=0, \, \dots, \, p_r(\bbx) =0, \,  q_1(\bbx) \neq 0 , \, \dots, \, q_s(\bbx) \neq 0\} 
\end{equation}
with polynomials $p_i(\bbx)$, $q_j(\bbx)$ in $\CC [\bbx]$ and integers $r$, $s \in \NN_0$. The solution set $\mathrm{Sol}(\algsystem)$ of the algebraic system \eqref{eqn:algebraicsystem} is defined as the set of all $\overline{\bbx} =(\overline{x}_1,\dots, \overline{x}_n) \in \CC^n$ satisfying the equations and inequations of $\algsystem$, that is as
\[
\mathrm{Sol}(\algsystem) = \{ \overline{\bbx} \in \CC^n \mid p_i(\overline{\bbx})=0, \, q_j(\overline{\bbx}) \neq 0 \ \mathrm{for} \ \mathrm{all} \ 1 \leq i \leq r \ \mathrm{and} \ 1 \leq j \leq s  \} .
\]
To fix ideas, we will use terminology from algebraic geometry, but this can be skipped without any loss by non-specialist readers. In algebraic geometry, Zariski closed sets, also called varieties, are sets of solutions of
systems of polynomial equations.
Zariski open sets are complements of closed sets (thus sets of solutions of systems of polynomial inequations) and
Zariski locally closed sets are
intersections between open and
closed sets.

Geometrically, $\mathrm{Sol}(\algsystem)$ is the difference of the two varieties 
\[
\{ \overline{\bbx} \in \CC^n \mid p_1(\overline{\bbx})=0, \, \dots, \, p_r(\overline{\bbx})=0  \} \ \mathrm{and} \ \{ \overline{\bbx} \in \CC^n \mid q_1(\overline{\bbx}) \cdots q_s(\overline{\bbx}) =0  \}
\]
and so it is a locally Zariski closed subset of $\CC^n$. 

In order to introduce the notion of an algebraic Thomas decomposition of a system $\algsystem$ we make the following definitions.
On the variables $\bbx=(x_1,\dots,x_n)$ of our polynomial ring $\CC [\bbx]$ we define a total ordering (sometimes also called a ranking) by setting 
$x_i < x_j$ for $i<j$. With respect to this ranking the leader $\mathrm{ld}(p(\bbx))$ of a non-constant polynomial $p(\bbx)$ is defined as the greatest variable appearing in $p(\bbx)$. In case $p(\bbx) \in  \CC$ is a constant polynomial, we set $\mathrm{ld}(p(\bbx))=1$. If we consider every polynomial $p(\bbx) \in \CC[\bbx]$ as a univariate polynomial in its leader, say $\mathrm{ld}(p(\bbx)) = x_k$, then the coefficients of $p(\bbx)$ as a polynomial in $x_k$ are polynomials in $\CC [ x_1,\dots, x_{k-1} ]$. The coefficient of the highest power of $\mathrm{ld}(p(\bbx))$ in $p(\bbx)$ is called the initial of $p(\bbx)$ which we denote by $\mathrm{init}(p(\bbx))$. The separant $\mathrm{sep}(p(\bbx))$ of a polynomial $p(\bbx)$ is defined as the partial derivative of $p(\bbx)$ with respect to its leader. 

\begin{definition}\label{def:simplesystem}
Let $\algsystem$ be the algebraic system of \eqref{eqn:algebraicsystem}. Then $\algsystem$ is called a simple algebraic system with respect to a ranking, if the following conditions are satisfied:
\begin{enumerate}
    \item The leaders of all equations and inequations are pairwise different, i.e.~we have 
    $$ \mathrm{card} \big(\{ \leader(p_1(\bbx)), \, \dots, \leader(p_r(\bbx)), \, \leader(q_1(\bbx)), \, \dots, \, \leader(q_s(\bbx)) \} \setminus \{1\} \big) =r+s   .
    $$
    This property is called triangularity.
    \item \label{def:simplesystem2} For every $p(\bbx) \in \{ p_1(\bbx), \, \dots, \, p_r(\bbx), \, q_1(\bbx), \, \dots, \, q_s(\bbx) \}$ the equation $\init(p(\bbx))=0$ has no solution in $\solset(\algsystem)$. We call this property non-vanishing initials.
    \item  For every $p(\bbx) \in \{ p_1(\bbx),\, \dots, \, p_r(\bbx), \, q_1(\bbx),\, \dots, \, q_s(\bbx) \}$ the equation $\mathrm{sep}(p(\bbx))=0$ has no solution in $\solset(\algsystem)$. This is called square-freeness.
\end{enumerate}
\end{definition}
The advantage of a simple algebraic system $\algsystem$ is that one can obtain its solution set by iteratively solving univariate polynomials. This is a consequence of the triangularity of a simple algebraic system. Indeed, the triangularity implies that there is at most one equation $p(\bbx)=0$ with leader $x_1$ or at most one inequation $q(\bbx)\neq 0$ with leader $x_1$. Note that $p(\bbx)$ or respectively $q(\bbx)$ is a univariate polynomial in $x_1$. The square-freeness implies that the number of zeros in $\CC$ of $p(\bbx)$ (respectively $q(\bbx)$) is equal the degree of $p(\bbx)$ (respectively $q(\bbx)$). In case of the equation $p(\bbx)=0$, any root $\overline{x}_1 \in \CC$ of 
$p(\bbx)$ can be chosen as the first coordinate of a solution $\overline{\bbx}$ of $\algsystem$. In case of the inequation $q(\bbx)\neq 0$, all elements of $\CC$ except for the roots of $q(\bbx)$ can here be chosen as the first coordinate of a solution. If there is no equation or inequation with leader $x_1$, then the first coordinate is free, that is $\overline{x}_1$ can be chosen arbitrary in $\CC$. Now we make the first iteration step. Again by triangularity there is at most one equation or inequation with leader $x_2$. If there is an equation or inequation with leader $x_2$, then we substitute $\overline{x}_1$ for $x_1$ in this equation or inequation and obtain so a univariate polynomial in $x_2$.  Condition \ref{def:simplesystem2} of Definition \ref{def:simplesystem} guarantees that the degree of the so obtained polynomial is independent of the choice of $\overline{x}_1$ and the square-freeness implies that the number of roots of this polynomial is equal to its degree. According to the three possible cases, that is there is an equation, inequation or neither of them, we determine as described above $\overline{x}_2 \in \CC$ for the second coordinate of a solution of $\algsystem$. An iteration of this process yields successively a solution $\overline{\bbx}=(\overline{x}_1,\dots,\overline{x}_n) \in \CC^n$ of $\algsystem$. Moreover, any solution of the simple algebraic system $\algsystem$ can be obtained by this process.    
\begin{definition}
    A Thomas decomposition of an algebraic system $\algsystem$ as in \eqref{eqn:algebraicsystem} consists of finitely many simple algebraic systems 
    $\algsystem_1,\dots , \algsystem_m$ such that $\solset(\algsystem)$ is the disjoint union of $\solset(\algsystem_1),\dots, \solset(\algsystem_m)$.
\end{definition}
It was proved by Thomas in \cite{th:sr,th:ds} that any algebraic system has a Thomas decomposition which is in general not unique. A Thomas decomposition can be determined algorithmically (see \cite{bglr:thomasalg}) and there is an implementation in \textsc{Maple}. 
A description of the implementation can be found in \cite{bl:thomasimpl,glhr:tdds}.

We will compute in the subsequent sections a Thomas decomposition of the symmetrized systems $M2$, $M3$, $M4$, $M8$ and $M9$ ($N=3$) by applying the \textsc{Maple} implementation to them. To this end we need to define a ranking on the polynomial ring 
\[
\CC [k_1,k_2,k_3,k_4,k_5,L_1,L_2,L_3,S_1,S_2].
\]
To simplify notation we collect the variables into $\bbk=(k_1,k_2,k_3,k_4,k_5)$ and $\bbv=(L_1,L_2,L_3,S_1,S_2)$. 
Since we want to solve the symmetrized systems for the variables $\bbk$, we rank them always higher than the collection of variables $\bbv$. Among the variables $\bbk$ we sometimes change the ranking for the different symmetrized systems. This is done to minimize the output of the Thomas decomposition and has no deeper meeting. The ranking of the variables 
$\bbk > \bbv$ implies that each simple algebraic system returned by the Thomas decomposition yields solutions for the variables $\bbk$ which are only valid for solutions of the variables $\bbv$ satisfying equation and inequation conditions in $\bbv$ over $\CC$.

\subsection{A Thomas decomposition for model 2}

The model studied here is a particular case of the one from Section~\ref{vier} and Theorem~\ref{interval}. We are going to determine the solutions of the algebraic system of equations 
\[
\begin{array}{rcl}
\algsystem &=& \big\{ -S_1 = k_5, \
-S_2 = k_5 \, (L_1 + k_1 + k_2 + k_3), \
L_1  = -k_1 - k_2 - k_3 - k_4 - k_5,   \\[0.5em] &&  \   
L_2  = k_1 \,k_4 + k_1 \, k_5 + k_2 \, k_3 + k_2 \, k_5 + k_3 \, k_4 + k_3 \,k_5, \
L_3 = -k_2 \, k_3 \, k_5  \big\} .
\end{array}
\]
To this end we compute with \textsc{Maple} an algebraic Thomas decomposition of $\algsystem$ with respect to the ranking 
$$k_1 > k_3 > k_2 > k_4 > k_5 >\bbv .$$ 
\textsc{Maple} returns eleven simple systems $\algsystem_1, \dots, \algsystem_{11}$. We will only compute here the solutions in $\bbk$ for the first simple systems, since it the most generic one, meaning that the solutions are valid for specializations of the variables $\bbv$ satisfying only inequations, i.e~the solutions for $\bbk$ are valid over an Zariski open subset of $\CC^5$. The first simple system consists of the equations and inequations
  \[
  \begin{array}{rcl}
 \algsystem_1 &=& 
 \big\{ L_1^2 \, S_1^3 \, S_2 + L_2^2 \, S_1^3 + S_1 \, S_2^3 + L_3^2 \, S_1 + (-S_1^2 \, S_2^2 + S_2^3 + 
       (S_1^3 \, S_2 - S_1 \, S_2^2) \, L_1  \\[0.5em] && \
       + (-S_1^4 + S_1^2 \, S_2) \, L_2 + (S_1^3 - S_1 \, S_2) \, L_3) \, k_1 + (-2 \, S_1^2 \, S_2^2 + (-S_1^4 - S_1^2 \, S_2) \, L_2  \\[0.5em] && \
       + (S_1^3 + S_1 \, S_2) \, L_3) \, L_1 + (S_1^3 \, S_2 - 2 \, L_3 \, S_1^2 + S_1 \, S_2^2) \, L_2 + (S_1^4 - 3 \, S_1^2 \, S_2) \, L_3 = 0,\\[0.5em] && \
        (L_1 \, S_1 \, S_2 - L_2 \, S_1^2 + L_3 \, S_1 - S_2^2) \, k_3 + (-S_1^2 + S_2) \, L_3 = 0, \\[0.5em] && \
        -L_1 \, S_1 \, S_2 + L_2 \, S_1^2 - L_3 \, S_1 + S_2^2 + (S_1^3 - S_1 \, S_2) \, k_2 = 0, \ 
        k_4 \, S_1 - S_1^2 + S_2 = 0, \\[0.5em] &&  \
        k_5 + S_1 = 0, \ 
        L_1 \, S_1 \, S_2 - L_2 \, S_1^2 + L_3 \, S_1 - S_2^2 \neq 0,  \
        S_1^3 - S_1 \, S_2 \neq 0, \
        S_2 \neq 0 \big\}.
  \end{array}
\]
One easily checks that the last three inequations do not involve any variable $\bbk$. Thus the inequations 
\begin{equation}\label{eqn:model2_simple1_cond}    
L_1 \, S_1 \, S_2 - L_2 \, S_1^2 + L_3 \, S_1 - S_2^2 \neq 0,  \quad S_1^3 - S_1 \, S_2 \neq 0, \quad S_2 \neq 0 
\end{equation}
define the above described Zariski open subset. Solving now successively the remaining equations for the variables $\bbk$ (they are all linear in their respective leaders), we obtain the solutions
  \begin{align*}
 k_1 &= \frac{-1}{(-S_1^2 \, S_2^2 + S_2^3 + 
       (S_1^3 \, S_2 - S_1 \, S_2^2) \, L_1  
       + (-S_1^4 + S_1^2 \, S_2) \, L_2 + (S_1^3 - S_1 \, S_2) \, L_3)}\\
       & (L_1^2 \, S_1^3 \, S_2 + L_2^2 \, S_1^3 + S_1 \, S_2^3 + L_3^2 \, S_1 + (-2 \, S_1^2 \, S_2^2 + (-S_1^4 - S_1^2 \, S_2) \, L_2 + 
      (S_1^3 + S_1 \, S_2) \, L_3) \, L_1  \\
      & +(S_1^3 \, S_2 - 2 \, L_3 \, S_1^2 + S_1 \, S_2^2) \, L_2  + (S_1^4 - 3 \, S_1^2 \, S_2) \, L_3),\\
       k_3 & = \frac{ (S_1^2 - S_2) \, L_3}{ (L_1 \, S_1 \, S_2 - L_2 \, S_1^2 + L_3 \, S_1 - S_2^2)}, \\
         k_2 &= \frac{L_1 \, S_1 \, S_2 - L_2 \, S_1^2 + L_3 \, S_1 - S_2^2}{S_1^3 - S_1 \, S_2}, \\
        k_4  & = \frac{S_1^2 - S_2}{ S_1}, \\
        k_5 & = - S_1, 
  \end{align*}
which are only valid for those $\overline{\bbv} \in \CC^5$ satisfying the inequations in \eqref{eqn:model2_simple1_cond}. If one is now interested in a solutions for $\bbk$ with respect to a $\overline{\bbv} \in \CC^5$ which does not satisfy the inequations of \eqref{eqn:model2_simple1_cond}, then there is a simple system among $\algsystem_2, \dots, \algsystem_{11}$, which one can solve successively for $\bbk$ as described in the previous subsection. The remaining simple systems $\algsystem_2, \dots, \algsystem_{11}$ can be found in subsection \ref{appendixmodel2} of the appendix. 

The conditions \eqref{eqn:model2_simple1_cond} do not contain equalities,
therefore the inverse problem has solutions on an open domain of dimension 5. According to the Section~\ref{sec:solv} this means that
the model $M2$ is solvable.

\subsection{A Thomas decomposition for model 3} 
The symmetrized algebraic system for model 3 is 
 \[
\begin{array}{rcl}
\algsystem &=& \big\{ 
-S_1 = k_5, \ 
-S_2 = k_5 \, (L_1 + k_1 + k_2 + k_3), \  
L_1  = -k_1 - k_2 - k_3 - k_4 - k_5,  \\[0.5em]
&& \ L_3  = -k_2 \, k_3 \, k_5, \ L_2  = k_1 \, k_4 + k_1 \, k_5 + k_2 \,k_3 + k_2\,k_4 + k_2\,k_5 + k_3\,k_4 + k_3\,k_5 \big\}
\end{array} 
\]
and we compute an algebraic Thomas decomposition for it with respect to the ranking 
$$k_1> k_2> k_3 > k_4 > k_5 > \bbv .$$
\textsc{Maple} returns $10$ simple systems $\algsystem_1,\dots ,\algsystem_{10}$. One easily checks by comparing the number of equations only involving the variables $\bbv$ (see appendix \ref{appendixmodel3} for the remaining simple systems $\algsystem_2,\dots ,\algsystem_{10}$), that the first simple system
\[
\begin{array}{rcl}
\mathcal{S}_1 &=& \big\{ k_1 \, k_3 \, S_1 \, S_2 + k_3^2 \, S_1 \, S_2 + L_3 \, S_2 + (L_2 \, S_1^2 - L_3 \, S_1) \, k_3 = 0, \ k_2 \, k_3 \, S_1 - L_3 = 0, \ k_3 \neq 0, \\[0.5em]
 & & \  k_4 \, S_1 - S_1^2 + S_2 = 0, \ k_5 + S_1 = 0, \ L_1 \, S_1 \, S_2 - L_2 \, S_1^2 + L_3 \, S_1 - S_2^2 = 0, \ S_1 \neq 0, \ S_2 \neq 0 \big\}
\end{array}
\]
is the most generic one. Here, even in the most generic case, the solutions for $\bbk$ are only valid for $\overline{\bbv} \in \CC^5$ lying on a Zariski locally closed subset defined by
\begin{equation}\label{eqn:model3_simple1_cond}
 L_1 \, S_1 \, S_2 - L_2 \, S_1^2 + L_3 \, S_1 - S_2^2 = 0, \ S_1 \neq 0, \, S_2 \neq 0 .
 \end{equation}
Solving again the remaining equations successively for $k_5$, $k_4$, $k_3$, $k_2$, $k_1$, we obtain the solutions 
\begin{align*}
    k_5 &= -S_1 \\
    k_4 &= \frac{S_1^2-S_2}{S_1} \\
    k_3 & \quad \mathrm{arbitrary} \, \mathrm{in} \, \CC \setminus \{ 0 \}, \\
    k_2 & = \frac{L_3}{k_3 \, S_1}, \\
    k_1 & = \frac{-1}{k_3 \, S_1 \, S_2} (k_3^2 \, S_1 \, S_2 + L_3 \, S_2 + (L_2 \, S_1^2 - L_3 \, S_1) \, k_3),
\end{align*}
where the expression for $k_2$ and $k_1$ depend on the choice made for $k_3$.

The conditions \eqref{eqn:model3_simple1_cond} contain one equality,
therefore the inverse problem has solutions on an open domain of dimension 4.
Furthermore, the set of solutions is infinite. 
According to the Section~\ref{sec:solv} this means that
the model $M3$ is not solvable.

\subsection{A Thomas decomposition for model 4}
In case of model 4 the symmetrized algebraic system is 
\[
\begin{array}{rcl}
\algsystem &=& \big\{ 
-S_1 = k_5, \ 
-S_2 = k_5 \, (k_1 + k_2 + k_3 + L_1), \ 
L_1 = -k_1 - k_2 - k_3 - k_4 - k_5, \\[0.5em]
&& \  L_2 = k_1 \, k_3 + k_1 \, k_4 + k_1 \, k_5 + k_2 \, k_3 + k_2 \, k_5 + k_3 \, k_4 + k_3 \, k_5, \ 
L_3 = -k_1 \, k_3 \, k_5 - k_2 \, k_3 \, k_5
\big\} .
\end{array}
\]
Using the ranking
$$k_1 > k_2> k_3> k_4> k_5> \bbv$$ 
the \textsc{Maple} implementation of algebraic Thomas decomposition returns 21 simple systems (see subsection \ref{appendixmodel4} of the appendix for the systems $\algsystem_2, \dots, \algsystem_{21}$). The first simple
system 
\[
\begin{array}{rcl}
\mathcal{S}_1 &=& \big\{ L_1 \, S_1^2 - L_2 \, S_1 - S_1 \, S_2 + (S_1^2 - S_2) \, k_1 + (S_1^2 - S_2) \, k_3 + L_3 = 0, \\[0.5em] 
&& \ L_2 \, S_1^2 - L_1 \, S_1 \, S_2 - L_3 \, S_1 + S_2^2 + (S_1^3 - S_1 \, S_2) \, k_2 = 0, \ k_3^2 \, S_1 + (L_1 \, S_1 - S_2) \, k_3 + L_3 = 0, \\[0.5em] && \ k_4 \, S_1 - S_1^2 + S_2 = 0, \ k_5 + S_1 = 0, \ L_1^2 \, S_1^2 - 2 \, L_1 \, S_1 \, S_2 - 4 \, L_3 \, S_1 + S_2^2 \neq 0, \ L_3 \neq 0, \\[0.5em] && \ S_1^3 - S_1 \, S_2 \neq 0, \ S_2  \neq 0 \big\}.
\end{array}
\]
is the most generic one. Indeed, all inequations appearing in $\algsystem_1$ involve only the variables $\bbv$, namely 
\[
L_1^2 \, S_1^2 - 2 \, L_1 \, S_1 \, S_2 - 4 \, L_3 \, S_1 + S_2^2  \neq 0, \ L_3 \neq 0, \ S_1^3 - S_1 \, S_2  \neq 0, \ S_2  \neq 0, 
\]
and so the solutions for $\bbk$ are valid for all $\overline{\bbv} \in \CC^5$ of the Zariski open subset defined by these inequations.
We can now successively solve the remaining equations in $\algsystem_1$ for the variables $k_5$, $k_4$, $k_3$, $k_2$ and $k_1$, where the equations for $k_5$, $k_4$, $k_2$, $k_1$ are all linear in their respective leaders except for the equation with leader $k_3$, 
$$ k_3^2 \, S_1 + (L_1 \, S_1 - S_2) \, k_3 + L_3 = 0,$$ 
which is quadratic. Since we want to consider only real roots for $k_3$, we require that the discriminant is positive, that is we change the inequation $L_1^2 S_1^2 - 2 L_1 S_1 S_2 - 4 L_3 S_1 + S_2^2  \neq 0$ into $L_1^2 S_1^2 - 2 L_1 S_1 S_2 - 4 L_3 S_1 + S_2^2  > 0$. 
Solving successively the equations with leader $k_5$, $k_4$, $k_3$, $k_2$, $k_1$ using for $k_3$ the quadratic formula, we obtain the solutions
\begin{align*}
    k_5 &= -S_1 ,\\
    k_4 & =\frac{S_1^2-S_2}{S_1}, \\
    k_3^{1,2} &= - \frac{L_1 \, S_1- S_2 \pm \sqrt{L_1^2 \, S_1^2 - 2 \, L_1 \, S_1 \, S_2 - 4 \, L_3 \, S_1 + S_2^2} }{2 \, S_1}, \\
    k_2 &= \frac{L_1 \, S_1 \, S_2 - L_2 \, S_1^2 + L_3 \, S_1 - S_2^2 }{S_1^3 - S_1 \, S_2} ,\\
    k_1 &=  \frac{-L_1 \, S_1^2 + L_2 \, S_1 + S_1 \, S_2  - (S_1^2 - S_2) \, k_3^{1,2} - L_3 }{S_1^2 - S_2}
\end{align*}
subject to the conditions 
\begin{equation}\label{model4conditions}
L_1^2 \, S_1^2 - 2 \, L_1 \, S_1 \, S_2 - 4 \, L_3 \, S_1 + S_2^2  > 0, \ L_3 \neq 0, \ S_1^3 - S_1 \, S_2 \neq 0, \ S_2 \neq 0. 
\end{equation}
Note that the solution for $k_1$ depends on the choice of the root $k_3^{1,2}$.

The generic simple system provides a finite set of solutions of the inverse problem on the open domain of dimension 5
defined by \eqref{model4conditions}. 
According to the Section~\ref{sec:solv} this means that
the model $M4$ is solvable. 



\subsection{A Thomas decomposition for model 8}
In case of model 8 the symmetrized algebraic system is 
\[
\begin{array}{rcl}
\algsystem &=& \big\{ 
 -S_1 = k_5, \ -S_2 = k_5 \, (k_1 + k_2 + L_1), \ L_1 = -k_1 - k_2 - k_3 - k_4 - k_5,  \\[0.5em]
 && \ L_3 = -k_1 \, k_2 \, k_5, \ L_2 = k_1 \, k_2 + k_1 \, k_3 + k_1 \, k_4 + k_1 \, k_5 + k_2 \, k_3 + k_2 \, k_5 \big\} . 
\end{array}
\]
We compute the algebraic Thomas decomposition of $\algsystem$ with respect to the ranking 
$$k_1 > k_3 > k_4 > k_2 > k_5 > \bbv $$ 
and obtain from \textsc{Maple} twelve simple systems $\algsystem_1,\dots, \algsystem_{12}$. One easily checks by comparing the number of equations which have leader one of the variables of $\bbv$, that the first simple system  
\[
\begin{array}{rcl}
\mathcal{S}_1 &=&  \big\{ S_1 \, k_1 \, k_2 - L_3 = 0, \ k_3 \, k_2 \, S_1^2 + L_1 \, S_1 \, S_2 - L_2 \, S_1^2 + L_3 \, S_1 - S_2^2 + (-S_1^3 + S_1 \, S_2) \, k_2 = 0, \\[0.5em] && \ S_1^2 \, k_2 \, k_4 - L_1 \, S_1 \, S_2 + L_2 \, S_1^2 - L_3 \, S_1 + S_2^2 = 0, \ k_2^2 \, S_1 + (L_1 \, S_1 - S_2) \, k_2 + L_3 = 0, \\[0.5em] && \ k_5 + S_1 = 0, \ L_1^2 \, S_1^2 - 2 \, L_1 \, S_1 \, S_2 - 4 \, L_3 \, S_1 + S_2^2  \neq 0, \ L_3  \neq 0, \ S_1  \neq 0 \big\} 
\end{array}
\]
is the most generic one. Analogously as in case of model 4 one determines the solutions 
\begin{align*}
    k_5 &= -S_1, \\
    k_2^{1,2} &= -\frac{L_1 \, S_1- S_2 \pm \sqrt{L_1^2 \, S_1^2 - 2 \, L_1 \, S_1 \, S_2 - 4 \, L_3 \, S_1 + S_2^2} }{2 \,S_1}, \\
    k_4 &= \frac{L_1 \, S_1 \, S_2 - L_2 \, S_1^2 + L_3 \, S_1 - S_2^2}{ k_2^{1,2} \, S_1^2}, \\
    k_3 &= - \frac{-S_1^3 \, k_2^{1,2} + L_1 \, S_1 \, S_2 - L_2 \, S_1^2 + S_1 \, S_2 \, k_2^{1,2} + L_3 \, S_1 - S_2^2}{k_2^{1,2} \, S_1^2}, \\
    k_1 &= \frac{L_3}{S_1 \, k_2^{1,2}}
\end{align*}
of $\algsystem_1$ subject to the conditions 
\begin{equation}\label{model8conditions}
L_1^2 \, S_1^2 - 2 \, L_1 \, S_1 \, S_2 - 4 \, L_3 \, S_1 + S_2^2  > 0, \ L_3  \neq 0, \ S_1  \neq 0.
\end{equation}
Note that here the solution for $k_4$, $k_3$ and $k_1$ depend on the choice of the root $k_2^{1,2}$ and that we changed the inequation $L_1^2 \, S_1^2 - 2 \, L_1 \, S_1 \, S_2 - 4 \, L_3 \, S_1 + S_2^2\neq0$ for the discriminant into $L_1^2 \, S_1^2 - 2 \, L_1 \, S_1 \, S_2 - 4 \, L_3 \, S_1 + S_2^2>0$ to guarantee that the roots are real. The remaining simple systems are presented in subsection \ref{appendixmodel8} of the appendix.

The generic simple system provides a finite set of solutions of the inverse problem on the open domain of dimension 5
defined by \eqref{model8conditions}. 
According to the Section~\ref{sec:solv} this means that
the model $M8$ is solvable.

\subsection{A Thomas decomposition for model 9}
We compute an algebraic Thomas decomposition of the algebraic system 
\begin{eqnarray}\label{problem:model9}
\algsystem &=& \big\{ -S_1 = k_5, \ -S_2 = k_5 \, (k_1 + k_2 + L_1), \ L_1 = -k_1 - k_2 - k_3 - k_4 - k_5 ,  \notag \\[0.5em] && \ L_3 = -k_1 k_2 k_5, \ L_2 = k_1 \, k_2 + k_1 \, k_4 + k_1 \, k_5 + k_2 \, k_3 + k_2 \, k_5 \big\}
\end{eqnarray}
for model 9 with respect to the ranking  
\[
k_3 > k_1 > k_4 > k_2 > k_5 > \bbv . 
\]
The \textsc{Maple} implementation returns $15$ simple systems $\algsystem_1,\dots,\algsystem_{15}$, where the most general simple system is 
\[
\begin{array}{rcl}
\algsystem_1 &=& 
\big\{ L_1 \, S_1 \, S_2 - L_2 \, S_1^2 + L_3 \, S_1 - S_2^2 + (2 \, k_2 \, S_1^2 + L_1 \, S_1^2 - S_1 \, S_2) \, k_3 + (-S_1^3 + S_1 \, S_2) \, k_2 = 0, \\[0.5em] &&  \ -L_1 \, S_1^2 + L_2 \, S_1 + S_1 \, S_2 + (2 \, k_2 \, S_1 + L_1 \, S_1 - S_2) \, k_4 + (-S_1^2 + S_2) \, k_2 - L_3 = 0, \\[0.5em] && \ k_1 \, S_1 + k_2 \, S_1 + L_1 \, S_1 - S_2 = 0, \ k_2^2 \, S_1 + (L_1 \, S_1 - S_2) \, k_2 + L_3 = 0, \\[0.5em]  &&  \ k_5 + S_1 = 0, \ L_1^2 \, S_1^2 - 2 \, L_1 \, S_1 \, S_2 - 4 \, L_3 \, S_1 + S_2^2  \neq 0, \ L_3  \neq 0, \ S_1  \neq 0 \big\}
.
\end{array}
\]
Similar as for model 4 and 8 one determines the solutions
\begin{eqnarray}
    k_5 &=&-S_1, \notag \\ 
    k_2^{1,2} &=& -\frac{L_1 \, S_1 - S_2 \pm \sqrt{L_1^2 \, S_1^2 - 2 \, L_1 \, S_1 \, S_2 - 4 \, L_3 \, S_1 + S_2^2}}{2 \, S_1} ,\notag \\
    k_4 &= &\frac{L_1 \, S_1^2 + S_1^2 \, k_2^{1,2} - L_2 \, S_1 - S_1 \, S_2 - S_2 \, k_2^{1,2} + L_3}{2 \, k_2^{1,2} \, S_1 + L_1 \, S_1 - S_2} ,\notag\\
    k_1 &= &- \frac{k_2^{1,2} \, S_1 + L_1 \, S_1 - S_2}{S_1}, \notag \\ 
    k_3 &=& -\frac{-S_1^3 \, k_2^{1,2} + L_1 \, S_1 \, S_2 - L_2 \, S_1^2 + S_1 \, S_2 \, k_2^{1,2} + L_3 \, S_1 - S_2^2}{S_1 \, (2 \, k_2^{1,2} \, S_1 + L_1 \, S_1 - S_2)} \label{solutions9simple}
\end{eqnarray}
subject to the conditions 
\begin{equation}\label{model9conditions}
L_1^2 \, S_1^2 - 2 \, L_1 \, S_1 \, S_2 - 4 \, L_3 \, S_1 + S_2^2 > 0, \ L_3  \neq 0, \ S_1  \neq 0 .
\end{equation}
Note that the solutions for $k_4$, $k_1$, $k_3$ depend on the choice of the root $ k_2^{1,2}$ and to guarantee real roots we changed the inequation representing the discriminant $L_1^2 \, S_1^2 - 2 \, L_1 \, S_1 \, S_2 - 4 \, L_3 \, S_1 + S_2^2 \neq 0$ into $L_1^2 \, S_1^2 - 2 \, L_1 \, S_1 \, S_2 - 4 \, L_3 \, S_1 + S_2^2 > 0$.
The remaining 14 simple algebraic systems can be found in subsection \ref{appendixmodel9} of the appendix.

The generic simple system provides a finite set of solutions of the inverse problem on the open domain of dimension 5
defined by \eqref{model9conditions}. 
According to the Section~\ref{sec:solv} this means that
the model $M9$ is solvable.

\section{Model degeneracy effects}\label{sec:rashomon}


Several examples were used elsewhere \cite{ocinneide_non-uniqueness_1989}
to illustrate the model degeneracy of the phase-type distribution.
Indeed, given a phase-type distribution, one can have two or more distinct representations in terms of the same chain but with different rate parameters, and a Markov chain may provide representations of all distributions represented by a different chain.
In this section, we explore 
the degeneracy of multi-exponential phase-type distributions and illustrate this phenomenon with examples significant in biological applications.


\subsection{Discrimination between variant models}
For a fixed number of states $N$,  {there are} several variant models that predict exactly the same phase-type distribution and fit  {the data equally} well. The variant models can be generated automatically. For instance
in Figure~\ref{figure2} we list, up to permutation symmetries of the transition graphs, all the models that produce three exponential ($N=3$) phase-type distributions. Among these models, 
$M2$, $M4$, $M8$, and $M9$ are solvable.
One can not discriminate between any two of the four solvable models by using only time-to-event
data.
The non-solvable model $M3$ does not generate
any multi-exponential phase-type distribution, but only those distributions whose parameters satisfy \eqref{eqn:model3_simple1_cond}. Additionally, the inverse problem has an infinite set of solutions for the model $M3$. 
Although in practice we favor solvable models because in this
case the inverse problem is well posed, 
ultimately only experimental results
should be used to exclude non-solvable variants. 

Mathematically, it is interesting to classify  solvable models. One classification can be performed, based on the number of solutions. For instance in Figure~\ref{figure2} the inverse problem  has  {a} unique solution for model $M2$, whereas it has two distinct solutions for models $M4$, $M8$, $M9$. Model $M9$ is obviously symmetric with respect to  permutation of the 
vertices 1 and 2, which explains the two-fold degeneracy of the solutions, but $M4$ and $M8$ have
no obvious permutation symmetries. It is therefore interesting to identify possible algebraic
relations relating these three models together. These relations will be discussed in the next subsection.

 {To identify model variants experimentally, we aim to evaluate the potential of additional data, beyond time-to-event measurements, to differentiate between models.} Two natural candidates for  {these} data are the steady-state occupation probabilities of the Markov chain states and the lifetime of each state, defined as the average time spent in a state before transitioning. Depending on the application, experimental methods are available to estimate these quantities.
 We will discuss some of these methods in the 
Section~\ref{sec7} dedicated to transcriptional bursting.



\subsection{Symmetries and classification of solvable models}
Solvable models 
that generate the same phase-type distribution
must satisfy certain common conditions, which we outline in this section. These conditions 
 {allow us to identify}
mappings that relate the parameters of one solvable model to those of another.

Let us consider the matrix $\vect{{\tilde Q}}^{\mathrm{red}}$ that is obtained from $\vect{\tilde Q}(\vect{k})$ by 
setting $k_{2N+1}$ to zero. This matrix
is the dual generator of a
reduced Markov chain
describing the transitions between the 
states $1,\ldots,N$.
Then the steady state probabilities of the  states $1,\ldots,N$
can be computed as solutions of the system:
\begin{eqnarray}
\vect{{\tilde Q}}^{\mathrm{red}} \begin{pmatrix} p_1(\vect{k}) \\  p_2(\vect{k}) \\ \vdots \\ p_N(\vect{k}) \end{pmatrix} &=& 0, \notag \\
p_1(\vect{k}) + p_2(\vect{k}) + \ldots + p_N(\vect{k}) &=& 1. \label{eq:occupancies}
\end{eqnarray}
Let $T_i$ be the lifetime of the state $i$, where $1\le i \le N$. 
This can be calculated  from the diagonal elements of the matrix $\vect{{\tilde Q}}^{\mathrm{red}}$:
\begin{equation}\label{eq:lifetimes}
    T_i = -1/{\tilde Q}^{\mathrm{red}}_{ii} = 1 / \sum_{j\neq i}{\tilde Q^{\mathrm{red}}_{ij}}.
\end{equation}
The result of these calculations for $N=3$ is given in the Table~\ref{table:times}.

\begin{table}[h!]
  \begin{tabularx}{\linewidth}{l*{4}{>{\centering\arraybackslash}Y}}
    \toprule
Model & $T_1^{-1}$  & $T_2^{-1}$  & $T_3^{-1}$  & $P_{3}$ \\
    \midrule
$M2$  &  $k_1+k_2$ & $k_3$ &  $k_4$  &   $\frac{k_2k_3}{k_1k_4 + k_2 k_3 +  k_3k_4}$ 
\\  
$M4$  &  $k_1+k_2$ & $k_3$ &  $k_4$  &   $\frac{k_1k_3+ k_2k_3}{k_1k_3 + k_1k_4 + k_2 k_3 +  k_3k_4}$    \\  
$M8$  &  $k_1$   & $k_2$  & $(k_3+k_4)$  & $\frac{k_1k_2}{k_1k_2 + k_1k_3 + k_1 k_4 +  k_2k_3}$  \\
$M9$  &  $k_1$    & $k_2$  &  $(k_3+k_4)$  &  $\frac{k_1k_2}{k_1k_2 + k_1k_4 + k_2k_3}$ \\
  \bottomrule
  \end{tabularx}
\caption{Reciprocal lifetimes $T_i^{-1}$ and steady state occupation probability $P_3$ for the solvable models $M2$, $M4$, $M8$, $M9$. \label{table:times} }
\end{table}


The steady state probabilities and lifetimes allow us to formulate some general constraints on solvable models
as follows
\begin{proposition}\label{lem:constraints}
All solvable models that 
generate the same phase-type distribution have
\begin{enumerate}[label=(\roman*)]
\item \label{lem15i}
the same value of $k_{2N-1}$,
\item  \label{lem15ii}
the same value of the lifetime $T_N$ of the state $N$,
\item  \label{lem15iii}
the same value of the steady state probability $p_N$ of the state $N$. 
\end{enumerate}
\end{proposition}
\begin{proof}
\ref{lem15i} is a direct consequence of \eqref{eq:sym}.
From \eqref{eq:sym} and \eqref{vieta} it  follows:
$$
S_1 - \frac{S_2}{S_1} = - k_{2N-1} 
+ a_{N-1}(\vect{k}) - c_{N-2}(\vect{k}).
$$
From the definitions of $a_{N-1}(\vect{k})$ and $c_{N-2}(\vect{k})$ and
elementary properties of the determinant we find
$$
-a_{N-1}(\vect{k}) = -c_{N-2}(\vect{k}) +
{\tilde Q}^{\mathrm{red}}_{NN} - k_{2N-1}.
$$
It follows 
$$
S_1 - \frac{S_2}{S_1} =-{\tilde Q}^{\mathrm{red}}_{NN}= T_N^{-1},
$$
which implies \ref{lem15ii}. 

Let $\bar w$ be the mean value of the phase-type distribution. 
The mean number of  events on an interval
$[0,T]$ is $T/\bar w$. The same number is equal to 
$T p_{N} k_{2N-1}$ because the model goes to $N+1$ only from the 
state $N$ with an intensity $k_{2N-1}$. It follows that
$$
\bar w = \frac{1}{p_{N} k_{2N-1}},
$$
for all models.
We already know that $k_{2N-1}$ is the same for all models, 
therefore \ref{lem15iii} follows. 
\end{proof}
As can be seen in Figure~\ref{figure2}, $M9$ is a symmetrical model. 
The transformation $\sigma_{M9} : \RR^5 \to \RR^5$ defined
by $\sigma_{M9}(k_1,k_2,k_3,k_4,k_5) = (k_2,k_1,k_4,k_2,k_5)$,
that is equivalent to a permutation of the vertex labels $1$ and $2$, leaves the inverse problem
\eqref{problem:model9} invariant and therefore transforms any solution of the problem into
another solution.
As $\sigma_{M9}^2=e$, where $e$ is the identity, the symmetry group is isomorphic to $\ZZ_2$.
Solutions that are not symmetric
are then two-fold degenerated. 

In order to see the symmetry of the solutions we rewrite \eqref{solutions9simple} as 
\begin{eqnarray}
    k_5 &=-S_1, \notag \\ 
    k_1 &= -\frac{L_1 \, S_1 - S_2 \pm \sqrt{ (L_1 \, S_1 - S_2)^2 - 4 \, L_3 \, S_1 }}{2 \, S_1} ,\notag \\
    k_2 &= -\frac{L_1 \, S_1 - S_2 \mp \sqrt{(L_1 \, S_1 - S_2)^2 - 4 \, L_3 \, S_1 }}{2 \, S_1} ,\notag \\
    k_3 &=  \frac{1}{2S_1} \left[
     S_1^2 - S_2 \pm 
    \frac{L_1S_1^3 - S_1^2S_2 - 2L_2S_1^2 + L_1S_1S_2 + 2L_3S_1 - S_2^2}{
    \sqrt{L_1 S_1 - S_2)^2 - 4 L_3 S_1}
    }
    \right]
    ,\notag \\
    k_4 &=  \frac{1}{2S_1} \left[
     S_1^2 - S_2 \mp
    \frac{L_1S_1^3 - S_1^2S_2 - 2L_2S_1^2 + L_1S_1S_2 + 2L_3S_1 - S_2^2}{
    \sqrt{L_1 S_1 - S_2)^2 - 4 L_3 S_1}
    } 
     \right].
    \label{solutions9sym}
\end{eqnarray}
\eqref{solutions9sym} shows that each of the two solutions of the inverse problem is not symmetric,
whereas the set of two solutions is symmetric with respect to  $\sigma_{M9}$ that
transforms one solution into the other.

The models $M4$ and $M8$ have no symmetry of the transition graphs and no obvious algebraic
symmetry of the inverse problem. 
However, the following property shows that these models are tightly related to $M9$. 
\begin{proposition}\label{prop:mappings}
There are bijective rational maps $f_{M,M'}: D \subset \RR^5 \to D'\subset \RR^5$,
relating parameters $\vect{k}$ of model $M$ to parameters $\vect{k'}$ of the model $M'$
(i.e. $\vect{k'} = f_{M,M'}(\vect{k})$)
where $M,M'$ are any pair in the list $\{M4, M8, M9\}$, and $D,D'$ are open domains
containing all the real positive parameters resulting from the inverse problem of $M,M'$. 
$f_{M,M'}$ satisfy the following constraints
\begin{enumerate}[label=(\roman*)]
\item \label{prop16i}
The parameter $k_5$  is the same for all the models. 
\item \label{prop16ii}
The lifetimes of the states $1,2,3$
is the same for all the models. 
\item \label{prop16iii}
 {Under steady state conditions,} the probability to be in the state $3$ 
is the same for all the models.  
\end{enumerate}
Model $M$ with parameters  $\vect{k}$ has exactly the same phase-type distribution 
as model $M'$ with phase-type $\vect{k'}$. The property does not hold if 
the model $M2$ is included in the list. 
\end{proposition}
\begin{proof}
Consider that there are mappings relating parameters of  
solvable models that have the same phase-type distribution, in 
particular the same parameters of the phase type distributions. 
Then, \ref{prop16i} and \ref{prop16iii} follow from the Proposition~\ref{lem:constraints}. The same Proposition implies that $T_3$ is the same for all models.
Longer, but elementary calculations using the solutions of the
inverse problem obtained in the section 
allow us to check that $T_1$ and $T_2$ are the same for all the models 
in the list $M4$,$M8$,$M9$. However, this is no true 
if we include $M2$ in the list (this model has different values of $T_1$ and $T_2$).
Reciprocally, if we impose \ref{prop16i}, \ref{prop16ii}, \ref{prop16iii} one finds four equations relating
$\vect{k}$ and $\vect{k'}$. These equations, if they have solutions, define
the map $f_{M,M'}$. As a matter of fact, solutions exist and are easy to 
compute. 
For instance if $k_i$ are  the parameters of $M9$
corresponding to the parameters $k'_i$ of $M8$, then using the Table~\ref{table:times} the conditions \ref{prop16ii}, \ref{prop16iii} read
\begin{eqnarray}
k_1 &=& k'_1, \notag \\
k_2 &=& k'_2, \notag \\
k_3+k_4 &=& k'_3 + k'_4, \notag \\
\frac{k_1k_2}{k_1k_2 + k_1k_4 + k_2k_3} &=& \frac{k'_1k'_2}{k'_1k'_2 + k'_1k'_3 + k'_1k'_4 + k'_2k'_3}. \label{map1relations}
\end{eqnarray}

The solutions of system \eqref{map1relations} read
\begin{eqnarray}
k'_1 &=& k_1, \notag \\
k'_2 &=& k_2, \notag \\
k'_3 &=&  \frac{ k_3(k_2 -k_1)}{k_2}, \notag \\
k'_4 &=& \frac{k_1 k_3 + k_2 k_4}{k_2},  \label{map1}
\end{eqnarray}
and define the bijective rational map
$f_{M9,M8}: \{k_1> k_2>0,k_3>0,k_4>0 \} \to 
\{  k'_2> k'_1>0,k'_3>0,k'_4>\frac{k'_1 k'_3}{k'_2-k'_1} \}$.

Similarly, solving
\begin{eqnarray}
k_1 &=& k'_1 +k'_2,\notag \\
k_2 &=& k'_3, \notag \\
k_3+k_4 &=& k'_4, \notag \\
\frac{k_1k_2}{k_1k_2 + k_1k_4 + k_2k_3} &=& \frac{k'_1k'_3+k'_2k'_3}{k'_1k'_3 + k'_1k'_4 + k'_2k'_3 + k'_3k'_4}, \label{map2relations}
\end{eqnarray}
with respect to $\vect{k}'$ leads to the equations
%
\begin{eqnarray}
k'_1=(k_1k_4 - k_2k_4)/(k_3 + k_4),\notag \\
k'_2=(k_1k_3 + k_2k_4)/(k_3 + k_4),\notag \\
k'_3= k_2,\notag \\
k'_4 = k_3 + k_4,
\label{map2}
\end{eqnarray}
that define the map $f_{M9,M4}:
\{ k_1 > k_2 > 0, k_3 > 0, k_4>0    \}
\to
\{ k'_1 > k'_2 > 0,  k'_2 > k'_3 > 0, k'_4>0    \}$.
\end{proof}
\begin{remark}
The Proposition~\ref{prop:mappings} explains the two-fold  degeneracy of the 
solutions of the inverse problem for models $M4$,$M8$.
Indeed the solutions of $M4$ read as $f_{M9,M8}(\vect{k})$
and $f_{M9,M8}\circ \sigma_{M9}(\vect{k})$ where
$\vect{k}$ is a solution of $M9$. 
\end{remark}
\begin{remark}
The Proposition~\ref{prop:mappings} allows to compute the symmetries of the models $M8$ and $M4$.  The symmetry groups are isomorphic 
 to $\ZZ_2$ and generated by
$f_{M9,M8} \circ \sigma_{M9} \circ f^{-1}_{M9,M8}$, and
$f_{M9,M4} \circ \sigma_{M9} \circ f^{-1}_{M9,M4}$, respectively. 
\end{remark}
 {\begin{remark}
Of course, there are also mappings that relate the parameters of models M4, M8, or M9 to those of model M2, resulting in the same phase-type distribution; however, these mappings are many-to-one.
This is consistent with the
remark in Section~\ref{vier} that
the degree of the field
extension from the field generated by survival function parameters to the field generated by kinetic parameters 
is one for $M2$ and
larger than one for the other models. 
\end{remark}}





\begin{figure}[H]
\begin{center}
\includegraphics[scale=0.4]{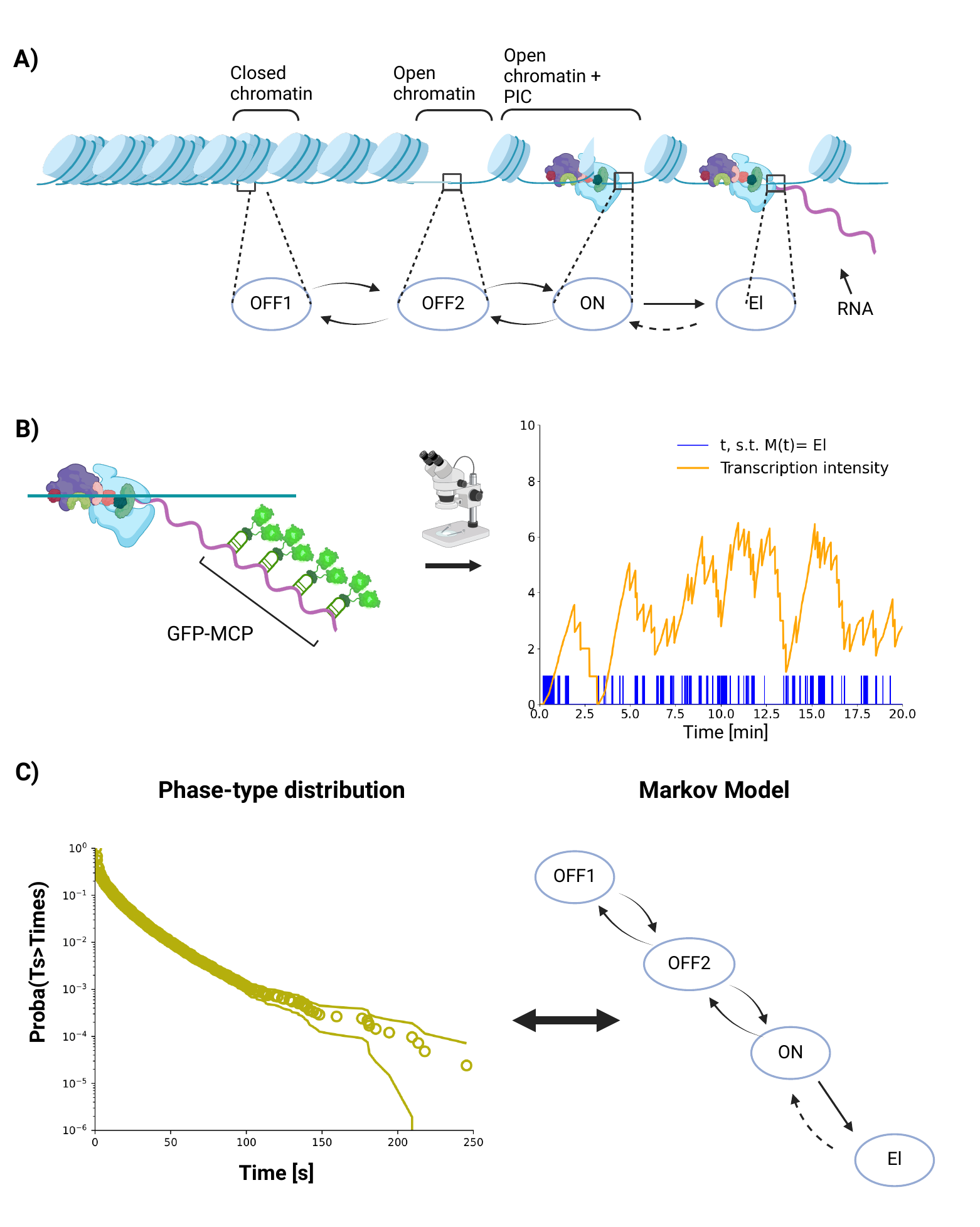}
\end{center}
\caption{Experimental setting for identifying 
phase-type distributions and Markov chain models of transcriptional bursting. A) Transcription requires multiple preceding steps
such as  chromatin opening and formation of transcription complexes needed for activation of the polymerase. These steps can be modeled
as transitions between discrete steps. 
B) Fluorescent tagging of mRNA structures allow
observation of mRNA synthesis. The observed signal is deconvolved to obtain the transcription initiation events. C) The phase-type distribution is used to infer the Markov chain model.  
\label{fig:experiment}
}
\end{figure}
\section{Application to transcriptional bursting models}
\label{sec7}
\subsection{Biological considerations and assumptions}
Transcription is the fundamental molecular 
process that copies the DNA into RNA. The synthesis of RNA  requires a dedicated molecular 
machinery, including an RNA polymerase (RNAP) enzyme 
that binds the DNA in specific regulatory regions able
to activate this process, called promoter regions. Transcription requires multiple preceding steps
such as  
chromatin opening and formation of transcription complexes needed for activation of the polymerase. 
Here we consider that transcription has a small number of limiting steps that can be modeled as transition between discrete states (Figure~
\ref{fig:experiment}~A and Table~\ref{table:models}).  We classify these states in three categories: productive ON states 
that can initiate transcription, non-productive OFF states 
that are unable to initiate or resume transcription, and
paused states PAUSE in which initiated transcription halts and can resume later or abort. 
\begin{table}[h!]
  \begin{tabularx}{\linewidth}{l*{4}{Y}}
    \toprule
    \multicolumn{5}{l}{\textbf{Possible state types  }} \\
    \midrule
Model & state $1$  & state $2$  & state $3$  & state $4$ \\

$M2$  &  OFF & OFF &  ON  &   EL 
\\  
$M3$  &  OFF & OFF &  ON  &  EL    \\ 
$M4$  &  PAUSE /OFF & OFF &  ON  &  EL    \\  
$M8$  &  PAUSE /OFF   & OFF  &  ON  & EL  \\
$M9$  &  PAUSE /OFF    & OFF/PAUSE  &  ON  &  EL \\
  \end{tabularx}
  \begin{tabularx}{\linewidth}{l*{2}{Y}}
    \toprule
    \multicolumn{2}{l}{\textbf{Possible biological interpretation of  state types}} \\
    \midrule
long OFF  &  Enhancer-Promoter disconnection  / Chromatin closing  \\

short OFF  & Transcription Factor unbinding  / Pre-initiation Complex  disassembly  \\

PAUSE & Promoter proximal pausing / Intrinsic pausing / Transcriptional roadblock \\

    \bottomrule
  \end{tabularx}

    \caption{Possible biological interpretation of the states in the models $M2$, $M3$, $M4$, $M8$, $M9$. For $M4,M8,M9$ pausing is always abortive. 
    \label{table:models} }
\end{table}
The promoter starts transcription in the productive ON state, when 
it can trigger several departures of RNAP molecules. After ON, the RNAP can eventually stop in PAUSE or commit to irreversible productive 
elongation modeled by the state EL.
Elongating polymerases move away from the promoter (to proceed into the gene body), thereby
freeing the promoter that instantly enters  
either a new 
productive ON state or a non-productive state OFF.
 
Live imaging of transcription  allows the  detection of 
the synthesis of new mRNA molecules  
 in real time and for each transcription site \cite{pichon2018growing}. Using the tool
BurstDeconv \cite{douaihy2023burstdeconv} we can deconvolve the live imaging signal and identify
the time of each elongation  event. This means that for each transcription site we observe the sequence of EL states (Figure~
\ref{fig:experiment}B). 
The other states of the promoter are  unobservable.

Strictly speaking, transcription initiation 
is not synonymous with the start of processive elongation, as transcription intiation can be abortive. Abortive transcription initiation events cannot be detected in our experimental settings.
Therefore, in this paper, whenever we refer to transcription initiation, we consider the processive case. 

We also consider that the state EL (denoted $N+1$ in our abstract Markov chain model) can be reached only from a state that can be either ON or PAUSE, and after reaching EL the promoter instantly switches to a unique return state $s$, that can be either 
productive or non-productive.
 {The second assumption is needed
for applying the phase-type distribution formalism.
The first assumption, not applying to the case of multiple ON states, was lifted in Subsection~\ref{sec:noC1}. 
Using this extension of our approach we can also address 
scenarios in which transcription occurs simultaneously on duplicate DNA templates, as for example newly replicated sister chromatids. 
}





\subsection{Discriminating between variant models using state probabilities and lifetimes.
}
As discussed in Section~\ref{sec:rashomon}, several different models can lead to the same phase-type distribution. For instance, three exponential 
phase-type distribution data can have the 
same likelihood for models $M2$, $M4$, $M8$, $M9$.
Furthermore, the states of these models 
can be interpreted in various ways. A possible
interpretation is provided in the Table~\ref{table:models}.






In this subsection we test the possibility to discriminate 
between solvable variant models using measurements of state lifetimes or probabilities. The probability that a
given state is occupied can be estimated from experimantal
data such as single molecule DNA footprinting assays
\cite{krebs2021studying,kleinendorst2021genome}.


Proposition~\ref{prop:mappings}
shows that state lifetimes can not discriminate between 
models $M4$,$M8$,$M9$. 

Additionally, Proposition~\ref{lem:constraints}
shows that for any $N$,  the occupancy
probability $p_N$ and the lifetime $T_N$ are 
the same for all models with $N$ states and the 
same phase-type
distribution. 

After these considerations, for $N=3$ 
we remain with three independent 
parameters $p_1,T_1,T_2$ as potential markers, discriminating between
variant solvable
models, except for models $M4$,$M8$,$M9$ where the only discriminating
parameter could be $p_1$. 

To evaluate the capacity of these markers to distinguish between variant models
we generate a dataset consisting of 
random parameters of the phase-type distribution 
$S_1,S_2,L_1,L_2,L_3$. For each set of phase-type distribution parameters we use the solutions of the inverse problem to compute the transition rate parameters
of the models $M2$,$M4$,$M8$,$M9$, excluding cases leading to
negative rate parameters.
From the rate parameters 
we compute the values of $p_1$, $T_1$, $T_2$
and the corresponding differences
$\Delta p_1$, $\Delta \log(T_1)$, $\Delta \log(T_2)$
between the maximum and minimum
values among the four variant models. Finally, we use these 
differences to estimate the
capacity of such marker values to discriminate between variants. 

We found that $p_1$ can discriminate between variants in $89\%$
of the cases, whereas $T_1,T_2$ can discriminate between variants
in $84\%$ of the cases, see Figure~\ref{figure5}. We have used random 
phase-type distribution parameters spanning domain values 
observed in studies of transcriptional bursting 
 of mammalian and fruit-fly genes \cite{Tantale2021,Pimmett2021}. 
This suggests that $p_1$ is better for discriminating among 
variant models than $T_1$ and $T_2$. This is confirmed
also for the particular case of a HIV gene promoter in mammalian
cells studied in \cite{Tantale2021}, where $T_1$ and $T_2$ are practically the same for all 
experimental conditions for the 3 variants, see Figure~\ref{figure7} below.
\begin{figure}
\begin{center}
\includegraphics[scale=0.8]{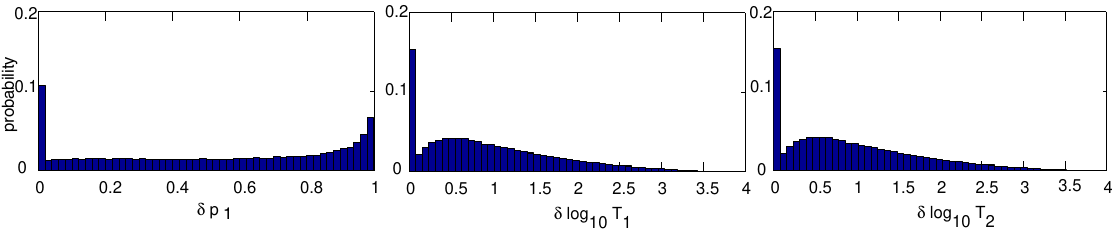}
\end{center}
\caption{
 {Testing the discrimination capacity of $p_1$, $T_1$, and $T_2$: We generated $100,000$ random parameter sets $(\lambda_1, \lambda_2, \lambda_3, A_1, A_2, A_3)$, where $A_1$ and $A_2$ are uniformly distributed in $[0,1]$, $A_3 = 1 - A_1 - A_2$, $\lambda_i < 0$, and $|\lambda_i|$ are log-uniformly distributed in $[10^{-4}, 1]$. From these, 61,411 sets were retained for which the inverse problem yielded positive rate parameters $k_1, k_2, k_3, k_4, k_5$ in at least one model (M2, M4, M8, or M9). Discrimination capacity for $p_1$, $T_1$, and $T_2$ is measured by the variations $\Delta p$, $\Delta \log_{10}(T_1)$, and $\Delta \log_{10}(T_2)$ across models with the same phase-type distribution parameters, where larger variations indicate better discrimination. The figures show frequency histograms of these variations, with no discrimination if the variation is zero. Results show that $p_1$, $T_1$, and $T_2$ fail to discriminate in $11\%$, $16\%$, and $16\%$ of the cases, respectively.}
\label{figure5}
}
\end{figure}

\begin{figure}[H]
\begin{center}
\includegraphics[scale=0.45]{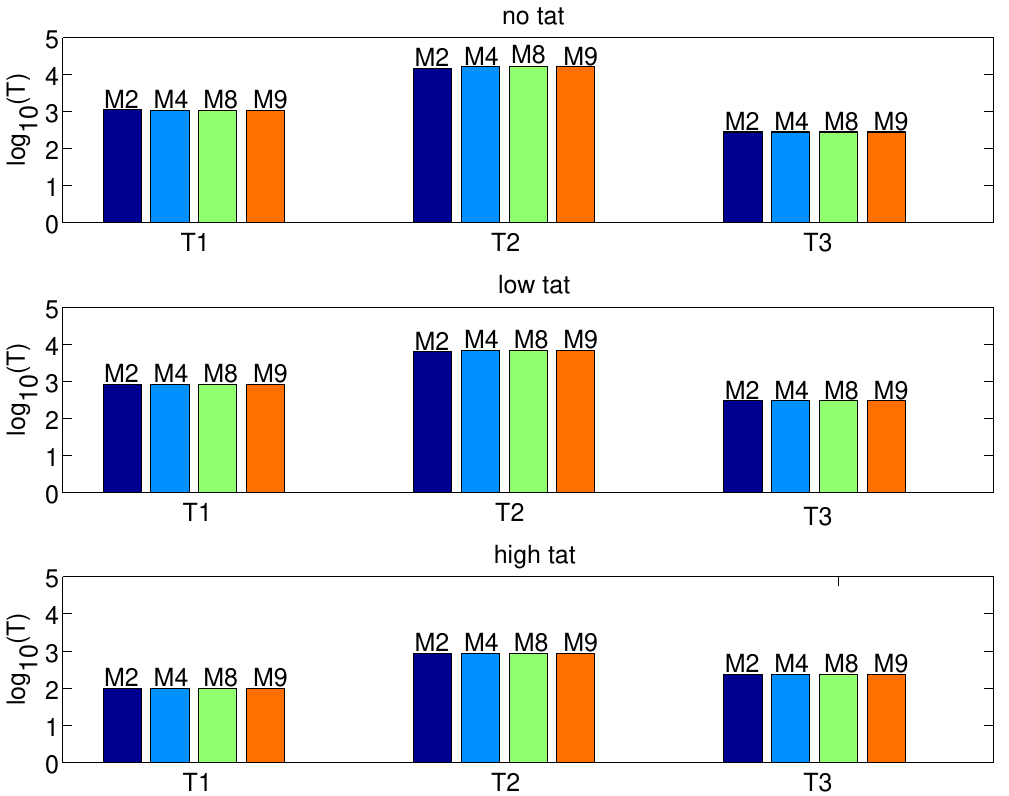}

\includegraphics[scale=0.45]{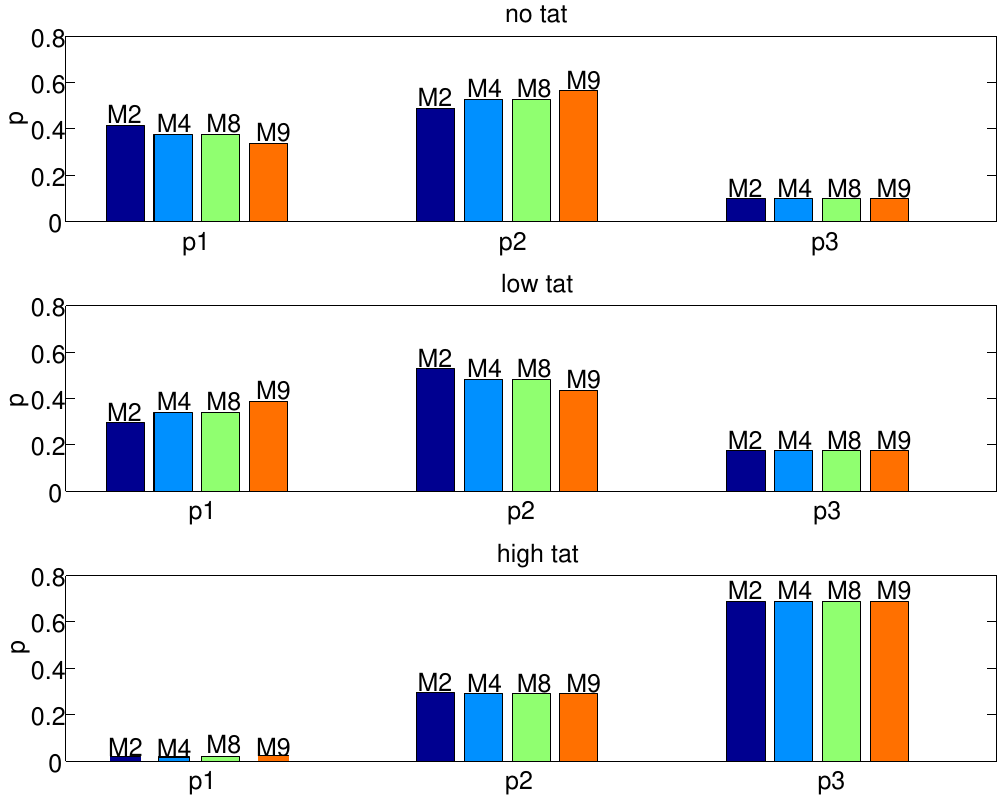}
\end{center}
\caption{Values of $p_i$ and $T_i$ for three  phase-type distributions corresponding to three experiments on cell lines infected with HIV. Increasing the level of the protein tat corresponds to decreasing the durations $T_1$, $T_2$
of the OFF states (see \cite{Tantale2021}). 
For these cells, under no-tat, low-tat, and high-tat conditions, the probabilities $p_1$, $p_2$ can discriminate between variant models, but the lifetimes $T_1$,$T_2$ cannot. \label{figure7}
}
\end{figure}
\section{Conclusion and perspectives.}
In contrast to statistical inference, which treats one model at a time, using symbolic solutions for the inverse phase-type distribution problem allows us to simultaneously infer all solvable models. Once the parameters of the phase-type distributions are estimated, our formulas provide the transition rate parameters for all solvable models.
 Thomas decomposition can be used to 
solve the inverse problem symbolically for models with a small number of states (we provide results for
$N=3$ but we tested the method also for $N=4$). 

The algebraic structure of the inverse problem for phase-type distributions leads us to a classification of Markov chain models compatible with the data. Some models are solvable, while others are not. Among the solvable three-state models, we discovered relationships stemming from the discrete  {symmetries} of the inverse problem. Extending these results to an arbitrary number of states requires different mathematical tools, which we will address in future work.

We identified examples of solvable models with an arbitrary number of states and established necessary conditions for their solvability. While sufficient conditions for solvability remain to be found, this task presents an exciting challenge for future research.


 {
In practical applications, Thomas decomposition provides formal solutions of the inverse problem. However, if a model has a solution set with a dimension greater than zero, or if multiple models have solutions, degeneracy arises, preventing unique identification from the experimental phase-type distribution. In such cases, additional data, such as state occupancy probabilities or state lifetimes, is needed. To achieve unique solutions, equations \eqref{vieta} and \eqref{eq:sym} (for $s=N$ or \eqref{vieta} and \eqref{eq:sym2} (for $s \neq N$) must be combined with equations \eqref{eq:occupancies} and \eqref{eq:lifetimes} that relate state occupancies and lifetimes to kinetic parameters.

However, limitations may arise from the complexity of the problem. Although models and parameters can theoretically be distinguished, many combinations may produce similar results, leading to practical degeneracy
and overfitting. Some practical
recipes to penalize complexity and avoid overfitting have been proposed in 
\cite{douaihy2023burstdeconv}.
}

We applied our approach to the study of transcriptional bursting in two different biological context, human cultured cells \cite{Tantale2021} and {\em Drosophila} developing embryos \cite{Pimmett2021}.

By mapping the parameters of the phase-type distribution to kinetic model parameters, we can identify mechanisms impacting transcriptional bursting from imaging datasets. This concept led to the development of BurstDeconv, a tool for analyzing transcriptional bursting data \cite{douaihy2023burstdeconv}. Utilizing Thomas decomposition, the computations presented here have enabled us to extend the range of models that BurstDeconv can analyze. Currently, BurstDeconv supports only the two-state random telegraph and the three-state $M2$ and 
$M9$ models. Additionally, we demonstrated that variant models indistinguishable by time-to-event data can be discriminated using state occupancy probabilities. 


 {In contrast to BurstDeconv, which only addresses transcription models with a single ON state, the results in this paper also apply  to models with multiple ON states. 
This broader approach enables the analysis of phenomena such as sister chromatids
(homologous DNA sequences from replication that can undergo transcription) where
transcription sites are considered to have multiple active ON states. 
In this case the symmetrized equations to solve are
\eqref{eq:symnoC1} (when the return
state is an ON state) and
\eqref{eq:sym2noC1} (when the return state is not an ON state).
 }

While our paper primarily centers on gene transcription applications, it's important to note that the phase-type distribution models discussed herein possess a broader range of potential uses. The most immediate of them concern the next steps of gene expression. Monitoring RNA translation initiation events as time-to-event data is feasible by using techniques such as SunTag 
\cite{tanenbaum2014protein,basyuk2021rna,blake2024imaging}, and the method presented here can also be applied to reverse engineer translation regulation mechanisms. Applying our multi-exponential models in epidemiology and medical research could also prove intriguing, as it may help discern disease stages and the trajectories of progression \cite{liquet2012investigating}.

\section*{Availability}
The implementation of the Thomas decomposition in Maple is available at

\noindent
\href{https://github.com/Computational-Systems-Biology-LPHI/Thomas_Decomposition/}{https://github.com/Computational-Systems-Biology-LPHI/Thomas$\_$Decomposition/}.

\section*{Acknowledgements}
OR and ED are founded by ANRS (ANRS0068).
MD is supported by a CNRS Prime80 grant to 
ML and OR. ML is sponsored by CNRS and ERC (SyncDev and LightRNA2Prot). We thank Werner Seiler for engaging and 
insightful  discussions, and John Reinitz for critical 
reading the manuscript.


\appendix

\section{Appendix: List of simple Thomas decomposition systems}

In the main text of the paper we have used the most
generic simple Thomas decomposition systems to solve the inverse problem. These most generic systems contain a minimum number of equations in $\vect{v}=(L_1,L_2,L_3,S_1,S_2)$ (no equation in $\vect{v}$ for solvable models). The remaining Thomas decomposition systems contain more equations in $\vect{v}$ and eventually
some equations in $\vect{k}$ are replaced by inequations.
Systems that contain only equations in
$\vect{v}$ and also some equations in
$\vect{v}$, lead to simpler formulas
for $\vect{k}$, valid for
special values of $\vect{v}$ satisfying 
equations (belonging to Zariski closed or locally closed subsets). Systems that contain  inequations 
in $\vect{k}$ lead to formulas where not 
all $k_i,\, 1\leq i \leq N$ are determined.

\subsection{The Remaining Simple Systems for Model 2}
\label{appendixmodel2}
\begin{align*}
\algsystem_2 &= \big\{ -L_2^2\,S_1^2 - L_3\,S_1^3 + 2\,L_2\,L_3\,S_1 - L_3^2 + (L_2\,S_1^3 - L_3\,S_1^2)\,k_1 + (L_2\,S_1^3 - L_3\,S_1^2)\,L_1 = 0,  \\[0.5em] \ 
&   \ L_3\,S_1 + (L_2\,S_1 - L_3)\,k_3 = 0, \ k_2\,S_1^2 + L_2\,S_1 - L_3 = 0, \ k_4 - S_1 = 0, \ k_5 + S_1 = 0, \\[0.5em] & \  L_2\,S_1 - L_3 \neq 0, \ S_1 \neq 0, \ S_2 = 0 \big\},\\[0.5em]
\algsystem_3 &= \big\{ k_1\,k_2\,S_1\,S_2 + k_2^2\,S_1\,S_2 + L_3\,S_2 + (L_2\,S_2 - L_3\,S_1)\,k_2 = 0, \ k_3\,k_2\,S_1 - L_3 = 0, \ k_2 \neq 0,  \\[0.5em] 
& \ k_4 = 0, \ k_5 + S_1 = 0, \ L_1\,S_1\,S_2 - L_2\,S_2 + L_3\,S_1 - S_2^2 = 0, \ S_1^2 - S_2 = 0, \ S_2 \neq 0 \big\}, \\[0.5em]
\algsystem_4 &= \big\{ k_1\,S_2 + k_3\,S_2 + L_2\,S_1 = 0, k_2 = 0, \ k_4\,S_1 - S_1^2 + S_2 = 0, \ k_5 + S_1 = 0, 
\\[0.5em] & \   L_1\,S_1\,S_2 - L_2\,S_1^2 - S_2^2 = 0, \ 
 L_3 = 0, \ S_1 \neq 0, \ S_2 \neq 0 \big\}, \\[0.5em]
\algsystem_5 &= \big\{ k_1 + k_3 + L_1 = 0, \ k_2 = 0, \ k_4 - S_1 = 0, \ k_5 + S_1 = 0, \ L_2 = 0, \ L_3 = 0, 
\\[0.5em] & \ S_1 \neq 0, \ S_2 = 0 \big\},\\[0.5em]
\algsystem_6 &= \big\{ k_1\,k_2 + k_2^2 + k_4^2 + k_4\,L_1 + (2\,k_4 + L_1)\,k_2 + L_2 = 0, \ k_3\,k_2 - k_2\,k_4 - k_4^2 - k_4\,L_1 - L_2 = 0,   \\[0.5em] & \  k_2 \neq 0, \  
k_5 = 0, \ L_3 = 0, \ S_1 = 0, \ S_2 = 0 \big\},\\[0.5em]
\algsystem_7 &= \big\{ k_1 + k_3 + k_4 + L_1 = 0, \ k_2 = 0, \ k_4^2 + k_4 \, L_1 + L_2 = 0, \ k_5 = 0, \ L_1^2 - 4\,L_2 \neq 0, \\[0.5em] & \ L_2 \neq 0, \ L_3 = 0, \ S_1 = 0, \ S_2 = 0 \big\}, \\[0.5em]
\algsystem_8 &= \big\{ k_1 + k_3 = 0, \ k_2 = 0, \ k_4 + L_1 = 0, \ k_5 = 0, \ L_1 \neq 0, \ L_2 = 0, \ L_3 = 0, \ S_1 = 0, \ S_2 = 0  \big\}, \\[0.5em]
\algsystem_9 &= \big\{ k_1 + k_3 + L_1 = 0, \ k_2 = 0, \ k_4 = 0, \ k_5 = 0, \ L_1 \neq 0, \ L_2 = 0, \ L_3 = 0, \ S_1 = 0, \ S_2 = 0 \big\}, \\[0.5em]
\algsystem_{10} &= \big\{ 2\,k_1 + 2\,k_3 + L_1 = 0, \ k_2 = 0, \ 2\,k_4 + L_1 = 0, \ k_5 = 0, \ L_1^2 - 4\,L_2 = 0, \ L_2 \neq 0, \ L_3 = 0, \\[0.5em] & \ S_1 = 0, \ S_2 = 0 \big\},\\[0.5em]
\algsystem_{11} &= \big\{ k_1 + k_3 = 0, \ k_2 = 0, \ k_4 = 0, \ k_5 = 0, \ L_1 = 0, \ L_2 = 0, \ L_3 = 0, \ S_1 = 0, \ S_2 = 0 \big\}.
\end{align*}

\subsection{The Remaining Simple Systems for Model 3}
\label{appendixmodel3}
\begin{align*}
 \algsystem_2 &= \big\{k_1 \, k_3 \, S_1 + k_3^2 \, S_1 + k_3 \, L_1 \, S_1 + L_3 = 0, \ k_2 \, k_3 \, S_1 - L_3 = 0, \ k_3 \neq 0, \ k_4 - S_1 = 0,  \\[0.5em] & \ k_5 + S_1 = 0, \ L_2 \, S_1 - L_3 = 0, \ S_1 \neq 0, \ S_2 = 0\}, \\[0.5em]
 \algsystem_3 &= \big\{ k_1 \, S_2 + k_2 \, S_2 + L_2 \, S_1 = 0, \ k_3 = 0, \ k_4 \, S_1 - S_1^2 + S_2 = 0, \ k_5 + S_1 = 0, \\[0.5em] & \ L_1 \, S_1 \, S_2 - L_2 \, S_1^2 - S_2^2 = 0, \ L_3 = 0, \ S_1 \neq 0, \ S_2 \neq 0 \big\},\\[0.5em]
\algsystem_4 &=  \big\{k_1 + k_2 + L_1 = 0, \ k_3 = 0, \ k_4 - S_1 = 0, \ k_5 + S_1 = 0, \ L_2 = 0, \ L_3 = 0, \ \\[0.5em] & \ S_1 \neq 0, \ S_2 = 0 \big\},\\[0.5em]
\algsystem_5 &= \big\{k_1 \, k_3 + k_3^2 + k_4^2 + k_4 \, L_1 + (k_4 + L_1) \, k_3 + L_2 = 0, \ k_2 \, k_3 - k_4^2 - k_4 \, L_1 - L_2 = 0, \\[0.5em] & \ k_3 \neq 0,  \ k_5 = 0, \ L_3 = 0, \ S_1 = 0, \ S_2 = 0 \big\}, \\[0.5em]
\algsystem_6 &= \big\{ k_1 + k_2 + k_4 + L_1 = 0, \ k_3 = 0, \ k_4^2 + k_4 \, L_1 + L_2 = 0, \ k_5 = 0, \ L_1^2 - 4 \, L_2 \neq 0,  \\[0.5em] & \ L_2 \neq 0, \ L_3 = 0, \ S_1 = 0,  \ S_2 = 0 \big\}, \\[0.5em]
\algsystem_7 &= \big\{ k_1 + k_2 = 0, \ k_3 = 0, \ k_4 + L_1 = 0, \ k_5 = 0, \ L_1 \neq 0, \ L_2 = 0, \ L_3 = 0, \ S_1 = 0, \\[0.5em] & \ S_2 = 0 \big\},\\[0.5em]
\algsystem_8 &= \big\{ k_1 + k_2 + L_1 = 0, \ k_3 = 0, \ k_4 = 0, \ k_5 = 0, \ L_1 \neq 0, \ L_2 = 0, \ L_3 = 0, \ S_1 = 0, \\[0.5em] & \ S_2 = 0 \big\},\\[0.5em]
\algsystem_9 &= \big\{ 2 \, k_1 + 2 \, k_2 + L_1 = 0, \ k_3 = 0, \ 2 \, k_4 + L_1 = 0, \ k_5 = 0, \ L_1^2 - 4 \, L_2 = 0, \ L_2 \neq 0, \\[0.5em] & \ L_3 = 0,  \ S_1 = 0, \ S_2 = 0 \big\},\\[0.5em]
\algsystem_{10} &= \big\{ k_1 + k_2 = 0, \ k_3 = 0, \ k_4 = 0, \ k_5 = 0, \ L_1 = 0, \ L_2 = 0, \ L_3 = 0, \ S_1 = 0, \ S_2 = 0 \big\}.
\end{align*}

\subsection{The Remaining Simple Systems for Model 4}
\label{appendixmodel4}
\begin{align*}
\algsystem_2 &= \big\{ L_1 \, S_1 \, S_2 - L_2 \, S_1^2 - S_2^2 + (S_1^3 - S_1 \, S_2) \, k_1 = 0, \ k_3 \, S_1 + L_1 \, S_1 - S_2 = 0, \\[0.5em] & \ -L_1 \, S_1 \, S_2 + L_2 \, S_1^2 + S_2^2 + (S_1^3 - S_1 \, S_2) \, k_2 = 0,  \ k_4 \, S_1 - S_1^2 + S_2 = 0, \ k_5 + S_1 = 0, \\[0.5em] & \ L_1 \, S_1 - S_2 \neq 0, \ L_3 = 0,  \ S_1^3 - S_1 \, S_2  \neq 0, \ S_2  \neq 0  \big\}, \\[0.5em]
\algsystem_3 &= \big\{ L_1 \, S_1^2 - L_2 \, S_1 - S_1 \, S_2 + (S_1^2 - S_2) \, k_1 = 0, \ -L_1 \, S_1 \, S_2 + L_2 \, S_1^2 + S_2^2 + (S_1^3 - S_1 \, S_2) \, k_2 = 0, \\[0.5em] & \ k_3 = 0, \ k_4 \, S_1 - S_1^2 + S_2 = 0, \ k_5 + S_1 = 0, \ L_1 \, S_1 - S_2  \neq 0, \ L_3 = 0, \ S_1^3 - S_1 \, S_2  \neq 0, \ S_2  \neq 0 \big\}, \\[0.5em]
\algsystem_4 &= \big\{-2 \, L_2 \, S_1^2 - S_1^2 \, S_2 + 2 \, L_3 \, S_1 - S_2^2 + (2 \, S_1^3 - 2 \, S_1 \, S_2) \, k_1 + (S_1^3 + S_1 \, S_2) \, L_1 = 0, \\[0.5em] & \ -L_1 \, S_1 \, S_2 + L_2 \, S_1^2 - L_3 \, S_1 + S_2^2 + (S_1^3 - S_1 \, S_2) \, k_2 = 0, \ 2 \, k_3 \, S_1 + L_1 \, S_1 - S_2 = 0,  \\[0.5em] & \ k_4 \, S_1 - S_1^2 + S_2 = 0, \ k_5 + S_1 = 0, \ L_1^2 \, S_1^2 - 2 \, L_1 \, S_1 \, S_2 - 4 \, L_3 \, S_1 + S_2^2 = 0, \\[0.5em] & \ L_3 \neq 0, \ S_1^3 - S_1 \, S_2  \neq 0, \ S_2 \neq 0 \big\}, \\[0.5em]
\algsystem_5 &= \big\{ -L_2 \, S_1 + (S_1^2 - S_2) \, k_1 = 0, \ L_2 \, S_1 + (S_1^2 - S_2) \, k_2 = 0, \ k_3 = 0, \ k_4 \, S_1 - S_1^2 + S_2 = 0, \\[0.5em] & \ k_5 + S_1 = 0, \ L_1 \, S_1 - S_2 = 0, \ L_3 = 0, \ S_1^3 - S_1 \, S_2 \neq 0, \  S_2  \neq 0 \big\}, \\[0.5em]
\algsystem_6 &= \big\{ k_1 \, S_1^2 + k_3 \, S_1^2 + L_1 \, S_1^2 - L_2 \, S_1 + L_3 = 0, \ k_2 \, S_1^2 + L_2 \, S_1 - L_3 = 0,  \ k_4 - S_1 = 0, \\[0.5em] & \ k_3^2 \, S_1 + k_3 \, L_1 \, S_1 + L_3 = 0, \ k_5 + S_1 = 0, \ L_1^2 \, S_1 - 4 \, L_3  \neq 0, \ L_3  \neq 0, \ S_1  \neq 0, \ S_2 = 0 \big\}, \\[0.5em]
\algsystem_7 &= \big\{ k_1 \, S_1 - L_2 = 0, \ k_2 \, S_1 + L_2 = 0, \ k_3 + L_1 = 0, \ k_4 - S_1 = 0, \ k_5 + S_1 = 0, \ L_1  \neq 0, \ L_3 = 0, \\[0.5em] & \ S_1 \neq 0, \ S_2 = 0 \big\}, \\[0.5em]
\algsystem_8 &= \big\{ k_1 \, S_1 + L_1 \, S_1 - L_2 = 0, \ k_2 \, S_1 + L_2 = 0, \ k_3 = 0, \ k_4 - S_1 = 0, \ k_5 + S_1 = 0, \ L_1  \neq 0,  \\[0.5em] & \ L_3 = 0, \ S_1  \neq 0, \ S_2 = 0 \big\}, \\[0.5em]
\algsystem_9 &= \big\{ 2 \, k_1 \, S_1^2 + L_1 \, S_1^2 - 2 \, L_2 \, S_1 + 2 \, L_3 = 0, \ k_2 \, S_1^2 + L_2 \, S_1 - L_3 = 0, \ 2 \, k_3 + L_1 = 0, \ k_4 - S_1 = 0, \\[0.5em] & \ k_5 + S_1 = 0, \ L_1^2 \, S_1 - 4 \, L_3 = 0, \ L_3  \neq 0, \ S_1  \neq 0, \ S_2 = 0 \big\}, \\[0.5em] 
\algsystem_{10} &= \big\{ k_1 \, S_1 - L_2 = 0, \ k_2 \, S_1 + L_2 = 0, \ k_3 = 0, \ k_4 - S_1 = 0, \ k_5 + S_1 = 0, \ L_1 = 0, \ L_3 = 0, \\[0.5em] & \ S_1  \neq 0, \ S_2 = 0 \big\}, \\[0.5em]
\algsystem_{11} &= \big\{ k_1 \, S_1 \, S_2 + k_2 \, S_1 \, S_2 + k_3 \, S_1 \, S_2 + L_2 \, S_2 - L_3 \, S_1 = 0, \ k_4 = 0, \ k_5 + S_1 = 0, \\[0.5em] & \ k_3^2 \, S_1 \, S_2 + L_3 \, S_2 + (L_2 \, S_2 - L_3 \, S_1) \, k_3 = 0,  \ L_1 \, S_1 \, S_2 - L_2 \, S_2 + L_3 \, S_1 - S_2^2 = 0, \\[0.5em] & \ L_2^2 \, S_2 - 2 \, L_2 \, L_3 \, S_1 - 4 \, L_3 \, S_1 \, S_2 + L_3^2 \neq 0, \ L_3  \neq 0, \ S_1^2 - S_2 = 0, \ S_2  \neq 0 \big\}, \\[0.5em]
\algsystem_{12} &= \big\{ k_1 + k_2 = 0, \ k_3 \, S_1 + L_2 = 0, \ k_4 = 0, \ k_5 + S_1 = 0, \ L_1 \, S_1 - L_2 - S_2 = 0, \ L_2 \neq 0,\\[0.5em] & \ L_3 = 0,  \ S_1^2 - S_2 = 0, \ S_2  \neq 0 \big\}, \\[0.5em]
\algsystem_{13} &= \big\{ k_1 \, S_1 + k_2 \, S_1 + L_2 = 0, \ k_3 = 0, \ k_4 = 0, \ k_5 + S_1 = 0, \ L_1 \, S_1 - L_2 - S_2 = 0, \ L_2 \neq 0, \\[0.5em] & \ L_3 = 0,  \ S_1^2 - S_2 = 0, \ S_2  \neq 0 \big\}, \\[0.5em]
\algsystem_{14} &= \big\{ 2 \, k_1 \, S_1 \, S_2 + 2 \, k_2 \, S_1 \, S_2 + L_2 \, S_2 - L_3 \, S_1 = 0, \ 2 \, k_3 \, S_1 \, S_2 + L_2 \, S_2 - L_3 \, S_1 = 0, \ k_4 = 0,  \\[0.5em] & \ k_5 + S_1 = 0, \ L_1 \, S_1 \, S_2 - L_2 \, S_2 + L_3 \, S_1 - S_2^2 = 0, \ L_2^2 \, S_2 - 2 \, L_2 \, L_3 \, S_1 - 4 \, L_3 \, S_1 \, S_2 + L_3^2 = 0,  \\[0.5em] & \ L_3  \neq 0, \ S_1^2 - S_2 = 0, \ S_2  \neq 0\}, \\[0.5em]
\algsystem_{15} &= \big\{ k_1 + k_2 = 0, \ k_3 = 0, \ k_4 = 0, \ k_5 + S_1 = 0, \ L_1 \, S_1 - S_2 = 0, \ L_2 = 0, \ L_3 = 0, \\[0.5em] & \ S_1^2 - S_2 = 0, \ S_2  \neq 0 \big\}, \\[0.5em]
\algsystem_{16} &= \big\{ k_1 \, k_4 - k_3^2 - k_3 \, L_1 - L_2 = 0, \ k_2 \, k_4 + k_3^2 + k_4^2 + k_4 \, L_1 + (k_4 + L_1) \, k_3 + L_2 = 0, \ k_4  \neq 0,  \\[0.5em] & \ k_5 = 0, \ L_3 = 0, \ S_1 = 0, \ S_2 = 0 \big\}, \\[0.5em]
\algsystem_{17} &= \big\{ k_1 + k_2 + k_3 + L_1 = 0, \ k_3^2 + k_3 \, L_1 + L_2 = 0, \ k_4 = 0, \ k_5 = 0, \ L_1^2 - 4 \, L_2  \neq 0, \ L_2  \neq 0, \\[0.5em] & \ L_3 = 0, \ S_1 = 0, \ S_2 = 0 \big\}, \\[0.5em]
\algsystem_{18} &= \big\{k_1 + k_2 = 0, \ k_3 + L_1 = 0, \ k_4 = 0, \ k_5 = 0, \ L_1  \neq 0, \ L_2 = 0, \ L_3 = 0, \ S_1 = 0, \ S_2 = 0 \big\}, \\[0.5em]
\algsystem_{19} &= \big\{ k_1 + k_2 + L_1 = 0, \ k_3 = 0, \ k_4 = 0, \ k_5 = 0, \ L_1 \neq 0, \ L_2 = 0, \ L_3 = 0, \ S_1 = 0, \ S_2 = 0 \big\}, \\[0.5em] 
\algsystem_{20} &= \big\{ 2 \, k_1 + 2 \, k_2 + L_1 = 0,  \ 2 \, k_3 + L_1 = 0, \ k_4 = 0, \ k_5 = 0, \ L_1^2 - 4 \, L_2 = 0, \ L_2  \neq 0, \\[0.5em] & \ L_3 = 0, \ S_1 = 0, \ S_2 = 0 \big\}, \\[0.5em]
\algsystem_{21} &= \big\{ k_1 + k_2 = 0, \ k_3 = 0, \ k_4 = 0, \ k_5 = 0, \ L_1 = 0, \ L_2 = 0, \ L_3 = 0, \ S_1 = 0, \ S_2 = 0 \big\}.
\end{align*}

\subsection{The Remaining Simple Systems for Model 8}
\label{appendixmodel8}
\begin{align*}
\algsystem_2 &= \big\{k_1 = 0, \ -L_1 \, S_1^2 + L_2 \, S_1 + S_1 \, S_2 + (L_1 \, S_1 - S_2) \, k_3 = 0, 
\\[0.5em] & \ L_1 \, S_1 \, S_2 - L_2 \, S_1^2 - S_2^2 + (L_1 \, S_1^2 - S_1 \, S_2) \, k_4 = 0, \ L_1 \, S_1 + S_1 \, k_2 - S_2 = 0, \ k_5 + S_1 = 0, \\[0.5em] & \ L_1 \, S_1 - S_2  \neq 0, \ L_3 = 0, \ S_1  \neq 0 \big\}, \\[0.5em]
\algsystem_3 &= \big\{(L_1 \, S_1 - S_2) \, k_1 + 2 \, L_3 = 0, \\[0.5em] & \ 2 \, L_2 \, S_1^2 + S_1^2 \, S_2 - 2 \, L_3 \, S_1 + S_2^2 + (L_1 \, S_1^2 - S_1 \, S_2) \, k_3 + (-S_1^3 - S_1 \, S_2) \, L_1 = 0,
\\[0.5em] & \ 2 \, L_1 \, S_1 \, S_2 - 2 \, L_2 \, S_1^2 + 2 \, L_3 \, S_1 - 2 \, S_2^2 + (L_1 \, S_1^2 - S_1 \, S_2)\, k_4 = 0, \ L_1 \, S_1 + 2 \, S_1 \, k_2 - S_2 = 0, \\[0.5em] & \ k_5 + S_1 = 0, \ L_1^2 \, S_1^2 - 2 \, L_1 \, S_1 \, S_2 - 4 \, L_3 \, S_1 + S_2^2 = 0, \ L_3  \neq 0, \ S_1  \neq 0 \big\},  \\[0.5em]
\algsystem_4 &= \big\{ L_2 \, S_1 + S_2 \, k_1 = 0, \ -S_1^2 + S_1 \, k_3 + S_1 \, k_4 + S_2 = 0, \ k_2 = 0, \ k_5 + S_1 = 0, \\[0.5em] & \ L_1 \, S_1 \, S_2 - L_2 \, S_1^2 - S_2^2 = 0, \ L_3 = 0, \ S_1  \neq 0, \ S_2  \neq 0 \big\},  \\[0.5em]
\algsystem_5 &= \big\{ k_1 + L_1 = 0, \ k_3 + k_4 - S_1 = 0, \ k_2 = 0, \ k_5 + S_1 = 0, \ L_2 = 0, \ L_3 = 0, \ S_1 \neq 0, \ S_2 = 0 \big\}, \\[0.5em]
\algsystem_6 &= \big\{k_1 + k_3 + k_4 + k_2 + L_1 = 0, \\[0.5em] & \ k_3^2 + k_4^2 + k_2^2 + k_2 \, L_1 + (2 \, k_4 + k_2 + L_1) \, k_3 + (2 \, k_2 + L_1) \, k_4 + L_2 = 0, \\[0.5em] & \ -L_1^2 + 2 \, L_1 \, k_2 + 3 \, k_2^2 + 4 \, k_2 \, k_4 + 4 \, L_2  \neq 0, \ k_2  \neq 0, \ k_5 = 0, \ L_3 = 0, \ S_1 = 0, \ S_2 = 0 \big\}, \\[0.5em]
\algsystem_7 &= \big\{ k_1 + k_3 + k_4 + L_1 = 0, \ k_3^2 + k_4^2 + k_4 \, L_1 + (2 \, k_4 + L_1) \, k_3 + L_2 = 0, \ k_2 = 0, \ k_5 = 0, \\[0.5em] & \ L_1^2 - 4 \, L_2  \neq 0, \ L_2  \neq 0, \ L_3 = 0, \ S_1 = 0, \ S_2 = 0 \big\}, \\[0.5em]
\algsystem_8 &= \big\{k_1 = 0, \ k_3 + k_4 + L_1 = 0, \ k_2 = 0, \ k_5 = 0, \ L_1 \neq 0, \ L_2 = 0, \ L_3 = 0, \ S_1 = 0, \ S_2 = 0 \big\},  \\[0.5em]
\algsystem_9 &= \big\{ k_1 + L_1 = 0, \ k_3 + k_4 = 0, \ k_2 = 0, \ k_5 = 0, \ L_1 \neq 0, \ L_2 = 0, \ L_3 = 0, \ S_1 = 0, \ S_2 = 0 \big\}, \\[0.5em]
\algsystem_{10} &= \big\{ 2\, k_1 + k_2 + L_1 = 0, \ L_1^2 - k_2^2 + 4 \, k_2 \, k_3 - 4 \, L_2 = 0, \\[0.5em] & \ -L_1^2 + 2 \, L_1 \, k_2 + 3 \, k_2^2 + 4 \, k_2 \, k_4 + 4 \, L_2 = 0,  \ k_2 \neq 0, \ k_5 = 0, \ L_3 = 0, \ S_1 = 0, \ S_2 = 0 \big\}, \\[0.5em]
\algsystem_{11} &= \big\{ 2 \, k_1 + L_1 = 0, \ 2 \, k_3 + 2 \, k_4 + L_1 = 0, \ k_2 = 0, \ k_5 = 0, \ L_1^2 - 4 \, L_2 = 0, \ L_2  \neq 0, \\[0.5em] & \ L_3 = 0, \ S_1 = 0, \ S_2 = 0 \big\},  \\[0.5em]
\algsystem_{12} &= \big\{k_1 = 0, \ k_3 + k_4 = 0, \ k_2 = 0, \ k_5 = 0, \ L_1 = 0, \ L_2 = 0, \ L_3 = 0, \ S_1 = 0, \ S_2 = 0 \big\}.
\end{align*}

\subsection{The Remaining Simple Systems for Model 9}
\label{appendixmodel9}
\begin{align*}
\algsystem_2 &= \big\{ -L_1 \, S_1^2 + L_2 \, S_1 + S_1 \, S_2 + (L_1 \, S_1 - S_2) \, k_3 = 0, \ k_1 = 0, 
\\[0.5em] & \ L_1 \, S_1 \, S_2 - L_2 \, S_1^2 - S_2^2 + (L_1 \, S_1^2 - S_1 \, S_2) \, k_4 = 0, \ k_2 \, S_1 + L_1 \, S_1 - S_2 = 0, \ k_5 + S_1 = 0, \\[0.5em] & \ L_1 \, S_1 - S_2  \neq 0, \ L_3 = 0, \ S_1  \neq 0 \big\}, \\[0.5em]
\algsystem_3 &= \big\{ L_1 \, S_1 \, S_2 - L_2 \, S_1^2 - S_2^2 + (L_1 \, S_1^2 - S_1 \, S_2) \, k_3 = 0, \ k_1 \, S_1 + L_1 \, S_1 - S_2 = 0, \\[0.5em] & \ -L_1 \, S_1^2 + L_2 \, S_1 + S_1 \, S_2 + (L_1 \, S_1 - S_2) \, k_4 = 0, \ k_2 = 0, \ k_5 + S_1 = 0, \ L_1 \, S_1 - S_2  \neq 0, \\[0.5em] & \ L_3 = 0, \ S_1 \neq 0 \big\}, \\[0.5em]
\algsystem_4 &= \big\{ k_3 \, S_1 + k_4 \, S_1 - S_1^2 + S_2 = 0, \ L_2 \, S_1 + (S_1^2 + S_2) \, k_1 - L_3 = 0,  \ k_5 + S_1 = 0, \\[0.5em] & \ L_2 \, S_1 + (S_1^2 + S_2) \, k_2 - L_3 = 0, \ -2 \, L_2 \, S_1^2 - S_1^2 \, S_2 + 2 \, L_3 \, S_1 - S_2^2 + (S_1^3 + S_1 \, S_2) \, L_1 = 0, \\[0.5em] & \ L_2^2 \, S_1^3 - 2 \, L_2 \, L_3 \, S_1^2 + L_3^2 \, S_1 + (-S_1^4 - 2 \, S_1^2 \, S_2 - S_2^2) \, L_3 = 0, \ L_3  \neq 0, \ S_1^3 + S_1 \, S_2 \neq 0, \ S_2  \neq 0 \big\}, \\[0.5em]
\algsystem_5 &= \big\{ k_3 \, S_1 + k_4 \, S_1 - S_1^2 + S_2 = 0, \ k_1 = 0, \ k_2 = 0, \ k_5 + S_1 = 0, \ L_1 \, S_1 - S_2 = 0, \ L_2 = 0, \\[0.5em] & \ L_3 = 0, \ S_1^3 + S_1 \, S_2  \neq 0, \ S_2  \neq 0 \big\}, \\[0.5em] 
\algsystem_6 &= \big\{ k_3 + k_4 - S_1 = 0, \ k_1 \, S_1^2 + L_2 \, S_1 - L_3 = 0, \ k_2 \, S_1^2 + L_2 \, S_1 - L_3 = 0, \ k_5 + S_1 = 0, \\[0.5em] & \ L_1 \, S_1^2 - 2 \, L_2 \, S_1 + 2 \, L_3 = 0, \ L_2^2 \, S_1^2 - L_3 \, S_1^3 - 2 \, L_2 \, L_3 \, S_1 + L_3^2 = 0, \ L_3  \neq 0, \ S_1  \neq 0, \ S_2 = 0 \big\}, \\[0.5em] 
\algsystem_7 &= \big\{ k_3 + k_4 - S_1 = 0, \ k_1 = 0, \ k_2 = 0, \ k_5 + S_1 = 0, \ L_1 = 0, \ L_2 = 0, \ L_3 = 0,  \ S_1  \neq 0, \\[0.5em] &\ S_2 = 0 \big\}, \\[0.5em] 
\algsystem_8 &= \big\{ k_3 \, S_1 + k_4 \, S_1 + 2 \, S_2 = 0, \ 2 \, k_1 \, S_1 + L_1 \, S_1 - S_2 = 0, \ 2 \, k_2 \, S_1 + L_1 \, S_1 - S_2 = 0, \ k_5 + S_1 = 0, \\[0.5em] & \ L_1^2 \, S_2 + 2 \, L_1 \, S_1 \, S_2 + 4 \, L_3 \, S_1 - S_2^2 = 0, \ L_2 \, S_2 + L_3 \, S_1 = 0, \ L_3  \neq 0, \ S_1^2 + S_2 = 0, \ S_2  \neq 0 \big\}, \\[0.5em]
\algsystem_9 &= \big\{ k_3 \, S_1 + k_4 \, S_1 + 2 \, S_2 = 0, \ k_1 = 0, \ k_2 = 0, \ k_5 + S_1 = 0, \ L_1 + S_1 = 0, \ L_2 = 0, \ L_3 = 0, \\[0.5em] & \ S_1^2 + S_2 = 0, \ S_2  \neq 0 \big\}, \\[0.5em] 
\algsystem_{10} &= \big\{ k_3 \, k_4 + k_4^2 + k_2^2 + k_2 \, L_1 + (2 \, k_2 + L_1) \, k_4 + L_2 = 0, \ k_1 \, k_4 - k_4 \, k_2 - k_2^2 - k_2 \, L_1 - L_2 = 0, \\[0.5em] & \ k_4  \neq 0, \ k_5 = 0, \ L_3 = 0, \ S_1 = 0, \ S_2 = 0 \big\}, \\[0.5em] 
\algsystem_{11} &= \big\{ k_3 + k_1 + k_2 + L_1 = 0, \ k_4 = 0, \ k_2^2 + k_2 \, L_1 + L_2 = 0, \ k_5 = 0, \ L_1^2 - 4 \, L_2  \neq 0, \ L_2  \neq 0, \\[0.5em] & \ L_3 = 0, \ S_1 = 0, \ S_2 = 0 \big\}, \\[0.5em] 
\algsystem_{12} &= \big\{ k_3 + k_1 = 0, \ k_4 = 0, \ k_2 + L_1 = 0, \ k_5 = 0, \ L_1  \neq 0, \  L_2 = 0, \ L_3 = 0, \ S_1 = 0, \ S_2 = 0 \big\}, \\[0.5em]
\algsystem_{13} &= \big\{ k_3 + k_1 + L_1 = 0, \ k_4 = 0, \ k_2 = 0, \ k_5 = 0, \ L_1  \neq 0, \ L_2 = 0, \ L_3 = 0, \ S_1 = 0, \ S_2 = 0 \big\}, \\[0.5em]
\algsystem_{14} &= \big\{ 2 \, k_3 + 2 \, k_1 + L_1 = 0, \ k_4 = 0, \ 2 \, k_2 + L_1 = 0, \ k_5 = 0, \ L_1^2 - 4 \, L_2 = 0, \ L_2  \neq 0, \\[0.5em] & \ L_3 = 0, \ S_1 = 0, \ S_2 = 0 \big\}, \\[0.5em] 
\algsystem_{15} &= \big\{ k_3 + k_1 = 0, \ k_4 = 0, \ k_2 = 0, \ k_5 = 0, \ L_1 = 0, \ L_2 = 0, \ L_3 = 0, \ S_1 = 0, \ S_2 = 0 \big\}.
\end{align*}

\end{document}